\documentclass[reqno]{amsart}

\usepackage{amssymb}
\usepackage{amsfonts}
\usepackage{amsmath}
\usepackage{color}
\usepackage{amsaddr}

\newtheorem{theorem}{Theorem}[section]

\newtheorem{corollary}[theorem]{Corollary}

\newtheorem{definition}[theorem]{Definition}

\newtheorem{lemma}[theorem]{Lemma}

\newtheorem{proposition}[theorem]{Proposition}

\theoremstyle{remark}

\newtheorem{remark}[theorem]{Remark}

\numberwithin{equation}{section}

\newcommand{\ep}{\varepsilon}

\DeclareMathOperator*{\essinf}{ess\,inf}

\begin{document}

\title[Global attractors for Nonlocal Cahn-Hilliard equations]{Upper-semicontinuity of the global attractors for a class of nonlocal Cahn-Hilliard equations}

\author[Joseph L. Shomberg]{Joseph L. Shomberg}

\subjclass[2010]{35B36, 37L30, 45K05, 74N20}

\keywords{Nonlocal Cahn-Hilliard equations, global attractors, upper-semicontinuity}

\address{Department of Mathematics and Computer Science, Providence
College, Providence, RI 02918, USA, \\
\tt{{jshomber@providence.edu} }}

\date{\today}

\begin{abstract}
The aim of this work is to examine the upper-semicontinuity properties of the family of global attractors admitted by a non-isothermal viscous relaxation of some nonlocal Cahn-Hilliard equations.
We prove that the family of global attractors is upper-semicontinuous as the perturbation parameters vanish. 
Additionally, under suitable assumptions, we prove that the family of global attractors satisfies a further upper-semicontinuity type estimate whereby the difference between trajectories of the relaxation problem and the limit isothermal non-viscous problem is explicitly controlled, in the topology of the relaxation problem, in terms of the relaxation parameters.
\end{abstract}

\maketitle 

\tableofcontents

\section{Introduction}

Inside a bounded domain (container) $\Omega\subset\mathbb{R}^3,$ we consider a phase separation model for a binary solution (e.g. a cooling alloy),
\[
\phi_t = \nabla\cdot[\kappa(\phi)\nabla\mu],
\]
where $\phi$ is the {\em{order-parameter}} (the relative difference of the two phases), $\kappa$ is the {\em{mobility function}} (which we set $\kappa\equiv1$ throughout this article), and $\mu$ is the {\em{chemical potential}} (the first variation of the free-energy $E$ with respect to $\phi$).
In the classical model,
\[
\mu = -\Delta\phi + F'(\phi) \quad \text{and} \quad E(\phi) = \int_\Omega \left( \frac{1}{2}|\nabla\phi|^2 + F(\phi) \right) dx,
\]
where $F$ describes the density of potential energy in $\Omega$ (e.g. the double-well potential $F(s)=\frac{1}{4}(1-s^2)^2$).

Recently the nonlocal free-energy functional appears in the literature \cite{Giacomin-Lebowitz-97},
\[
E(\phi) = \int_\Omega\int_\Omega \frac{1}{4}J(x-y)(\phi(x)-\phi(y))^2 dxdy + \int_\Omega F(\phi) dx,
\]
hence, the {\em{chemical potential}} is, $\mu = a\phi - J*\phi + F'(\phi),$ where
\begin{align}
a(x) = \int_\Omega J(x-y) dy \quad \text{and} \quad (J*\phi)(x) = \int_\Omega J(x-y)\phi(y) dy.  \notag
\end{align}

In this article we consider the following problems: for $\alpha>0$, $\delta>0$, and $\ep>0$ the {\em{relaxation}} Problem P$_{\alpha,\ep}$ is, given $T>0$ and $(\phi_0,\theta_0)^{tr},$ find $(\phi^+,\theta^+)^{tr}$ satisfying
\begin{eqnarray}
\phi^{+}_t = \Delta\mu^{+} &\text{in}& \Omega\times(0,T)  \label{rel-1} \\ 
\mu^{+} = a\phi^{+} - J*\phi^{+} + F'(\phi^{+}) + \alpha \phi^{+}_t - \delta\theta^{+} &\text{in}& \Omega\times(0,T)  \label{rel-2} \\ 
\ep\theta_t^{+} - \Delta\theta^{+} = -\delta \phi_t^{+} &\text{in}& \Omega\times(0,T)  \label{rel-3} \\ 
\partial_n\mu^{+} = 0 &\text{on}& \Gamma\times(0,T)  \label{rel-4} \\ 
\partial_n\theta^{+} = 0 &\text{on}& \Gamma\times(0,T)  \label{rel-5} \\ 
\phi^{+}(x,0) = \phi_0(x) &\text{at}& \Omega\times\{0\}  \label{rel-6} \\ 
\theta^{+}(x,0) = \theta_0(x) &\text{at}& \Omega\times\{0\}. \label{rel-7}  
\end{eqnarray}
Formally setting $\alpha=0$ and $\ep=0$ in the above equations we obtain the {\em{limit}} Problem P$_{0,0}$: given $T>0$ and $\phi_0,$ find $\phi^0$ satisfying
\begin{eqnarray}
(1+\delta^2)\phi^0_t = \Delta\mu^0 &\text{in}& \Omega\times(0,T)  \label{lim-1} \\ 
\mu^0 = a\phi^0 - J*\phi^0 + F'(\phi^0)  &\text{in}& \Omega\times(0,T)  \label{lim-2} \\ 
\partial_n\mu^0 = 0 &\text{on}& \Gamma\times(0,T)  \label{lim-4} \\ 
\phi^0(x,0) = \phi_0(x) &\text{at}& \Omega\times\{0\}.  \label{lim-6}
\end{eqnarray}

The main focus of this article is to examine the stability of the asymptotic behavior, via global attractors, when we allow both $\alpha\rightarrow0^+$ and $\ep\rightarrow0^+$.
For ease of presentation, throughout we assume there is $\delta_0>0$ so that $\delta\in(0,\delta_0]$, and also $(\alpha,\ep)\in(0,1]\times(0,1].$

Let us now give some preliminary words on the motivation for using nonlocal diffusion.
First, in \cite[Equation (0.2)]{AVMRTM10} the nonlocal diffusion terms $a\phi-J*\phi$ appear as,
\[
\int_\Omega J(x-y)\left( \phi(x,t)-\phi(y,t) \right)dy,
\]
i.e. $a(x)=J*1.$
Heuristically, this integral term ``takes into account the individuals arriving at or leaving position $x$ from other places.''
In this setting, the term $a(x)\ge0$ is a factor of how many individuals arrive at position $x$.
Since the integration only takes place over $\Omega,$ individuals are not entering nor exiting the domain.
Hence, this representation is faithful to the desired mass conservation law we typically associate with Neumann boundary conditions. 
Although Neumann boundary conditions for the chemical potential $\mu$ make sense from the physical point of view of mass conservation, it is not necessarily true that the interface between the two phases is always orthogonal to the boundary, which is implied by the boundary condition $\partial_{n}\phi=0$ which commonly appears in the literature. 
This is partially alleviated by using nonlocal diffusion on $\phi.$

There is obvious motivation already in the literature to investigate Problem P$_{\alpha,\ep}$ from the point of view of a singular limit of a Caginalp type phase-field system (cf. \cite[Equations (1.1)-(1.3)]{Gal&Grasselli08}, \cite[Equations (1.1)-(1.3)]{GGM08-2} and \cite{Miranville&Zelik02}).
Of the non-isothermal, nonlocal Allen-Cahn system, 
\begin{equation}  \label{mot-1}
\left\{ \begin{array}{l} \alpha \phi_t + a\phi-J*\phi + F'(\phi) = \delta\theta \\ 
\ep_1\theta_t - \Delta\theta = -\delta\phi_t, \end{array} \right. 
\end{equation}
with $\alpha>0$, $\delta>0$, and $\ep_1>0,$ the singular limit $\ep_1\rightarrow 0^+$ formally recovers the following isothermal, viscous, nonlocal Cahn-Hilliard equation,\begin{equation}  \label{mot-1.1}
\phi_t - \Delta(a\phi-J*\phi+F'(\phi)+\alpha\phi_t) = 0. 
\end{equation}
Equation \eqref{mot-1.1} in the case where $F$ is a singular (logarithmic) potential was studied in \cite{Gal&Grasselli14}. 
We should also notice that when we iterate this procedure to an appropriate non-isothermal version of \eqref{mot-1.1}, the resulting system is equivalent to \eqref{mot-1.1}. 
Indeed, when we consider the system,
\begin{equation*}
\left\{ \begin{array}{l} \phi_t = \Delta\mu \\ 
\mu = a\phi - J*\phi + F'(\phi) + \alpha\phi_t - \delta\theta \\ 
\ep_2\theta_t -\Delta\theta = -\delta\phi_t, \end{array} \right. 
\end{equation*}
the formal limit $\ep_2\rightarrow0^+$ yields the isothermal, viscous, nonlocal Cahn-Hilliard equation,
\[
\varphi_t=\Delta(a\phi-J*\phi+F'(\phi)+\beta\varphi_t),
\]
where 
\[
\beta=\frac{\alpha}{1+\delta^2} \quad \text{and} \quad \varphi(t)=\phi((1+\delta^2)t).
\]

A full treatment of well-posedness and global attractors and their regularity already appears in the literature.
In particular, for Problem P$_{0,0}$ see \cite{Gal&Grasselli14} and for Problem P$_{\alpha,\ep}$ see \cite{Shomberg-n16}.
Our main goal is to determine in what sense Problem P$_{\alpha,\ep}$ might converge to Problem P$_{0,0}$.
Such convergence results may have begun with the hyperbolic relaxation of a Chaffee--Infante reaction diffusion equation in \cite{Hale&Raugel88}.
The motivation for hyperbolic relaxation is that it alleviates the parabolic problems from the sometimes unwanted property of ``infinite speed of propagation''.
Hale and Raugel proved in \cite{Hale&Raugel88} the existence of a family of global attractors that is upper-semicontinuous in the phase space. 
A global attractor is a unique compact invariant subset of the phase space that attracts all trajectories of the associated dynamical system, even at arbitrarily slow rates (cf. \cite{Kostin98} and \cite[Theorem 14.6]{Robinson01}).
In a sense which will become clearer in Section \ref{s:upper-sc}, upper-semicontinuity guarantees the attractors to not ``blow-up'' as the perturbation parameter vanishes; i.e.,
\[
\sup_{x\in A^\ep}\inf_{y\in A^0}\|x-y\|_{X^\ep}\longrightarrow 0 \quad\text{as}\quad \ep\rightarrow 0^+.
\]
A complete treatment of the upper-semicontinuity of the global attractors admitted by the semiflow for the corresponding {\em local} problem appears in \cite{Gal&Miranville09}.
In many respects, the present work aims to emulate the continuity result found there.

Unlike global attractors, exponential attractors (sometimes called inertial sets) are compact positively invariant sets possessing finite fractal dimension that attract bounded subsets of the phase space exponentially fast (cf. \cite{EFNT95}). 
It can readily be seen that when both a global attractor $\mathcal{A}$ and an exponential attractor $\mathcal{M}$ exist, then $\mathcal{A}\subset \mathcal{M}$ provided that the basin of attraction of $\mathcal{M}$ is the entire phase-space, and so the global attractor is also finite dimensional. 
In this article we do not turn our attention to proving the existence of exponential attractors, however, we will be interested in certain convergence properties that may be possessed by families of exponential attractors. 

Robust families of exponential attractors (that is, both upper- and lower-semicontinuous with explicit control over semidistances in terms of the perturbation parameter) of the type reported in \cite{GGMP05} have appeared in numerous applications, of which we will limit ourselves to here mention only a few of those applications to Cahn-Hilliard equations and phase-field equations.
Most similar to our Problem P$_{\alpha,\ep}$ (a non-isothermal viscous relaxation), Gal and Miranville show in \cite{Gal&Miranville09-2} the existence of a family of exponential attractors that is robust (at zero) with respect to $\delta$ and $\ep$ for any $\alpha>0$ {\em{fixed}}.
They also establish robustness for $\alpha$, $\delta$ and $\ep$ at $0.$
The global well-posedness for their model is detailed in \cite{Gal08-2}.
Robust exponential attractors for an isothermal nonviscous Cahn-Hilliard equation with singularly perturbed boundary conditions appears in \cite{Gal08-3}.
The works \cite{Conti-Mola-08} and \cite{GMPZ10} which contains some applications of memory relaxation of reaction diffusion equations: Cahn--Hilliard equations, phase-field equations, wave equations, beam equations, and numerous others. 
The novelty here being the presence of an exponentially fading ``memory'' term appearing with a singularly perturbed kernel which converges to the Dirac delta function as the perturbation parameter vanishes.
These works (\cite{Conti-Mola-08,GMPZ10}) are also focused on proving the existence of a robust family of exponential attractors.
The hyperbolic relaxation of the 3D Cahn-Hilliard equation, i.e. 
\[
\ep\phi_{tt}+\phi_t-\Delta(-\Delta\phi+F'(\phi)+\alpha\phi)=0,
\]
is discussed in \cite{GGMP05-CH3D} where it is shown that the problem admits a family of exponential attractors, robust at $\alpha=\ep=0$.
For the interested reader, an analysis the 1D counterpart appears in \cite{GGMP05-CH1D}.
Finally, we recall from the above discussion that the viscous Cahn-Hilliard equations appears as a singular limit of a Caginalp type phase-field system.
Relaxation problems of this type were also shown to possess robust exponential attractors.
Indeed, we refer to \cite{Miranville&Zelik02} and \cite{GGM08-2}, the latter being subject to physically relevant dynamic boundary conditions.

Our interest in robustness is due to the fact that it typically relies on an estimate of the form, 
\begin{equation}  \label{robust-intro}
\|S_\varepsilon(t)x-\mathcal{L}S_{0}(t)\Pi x\|_{X^\varepsilon}\le C\varepsilon^p,
\end{equation}
for all $t$ in some interval, where $x\in X^\varepsilon$, $S_\varepsilon(t):X^\ep\rightarrow X^\ep$ and $S_{0}(t):X^0\rightarrow X^0$ are semigroups generated by the solutions of the perturbed problem and the limit problem, respectively, $\Pi$ denotes a projection from $X^\varepsilon$ onto $X^0$ and $\mathcal{L}$ is a ``lift'' from $X^0$ into $X^\varepsilon$, and finally $C,p>0$ are constants.
In obtaining our (direct) upper-semicontinuity type result (appearing in Section \ref{s:upper-sc}), controlling a difference of this type in a suitable norm is crucial. 
The estimate (\ref{robust-intro}) means we can approximate the limit problem with the perturbation with control explicitly written in terms of the perturbation parameter. 
Usually such control is only exhibited on compact time intervals. 
For the model problems under consideration here, the right-hand side of the corresponding difference will be controlled in terms of the perturbation parameters $\alpha$ and $\ep$, and on compact time intervals, but at a cost of restricting the size of two other structural parameters.

In the next section we provide the functional framework behind Problem P$_{0,0}$ and Problem P$_{\alpha,\ep}$. 
Section \ref{s:models} is devoted to recalling several important aspects of Problem P$_{0,0}$ and Problem P$_{\alpha,\ep}$ such as (global) well-posedness, dissipation, and the existence of global attractors.
The upper-semicontinuity results appear in Section \ref{s:upper-sc}. 
The main points of this article are as follows:

\begin{itemize}

\item We prove the family of global attractors admitted by Problem P$_{\alpha,\ep}$ and Problem P$_{0,0}$ is upper-semicontinuous as the perturbation parameters $\alpha$, $\ep$ vanish.
Here we rely on the classical proof in \cite{Hale&Raugel88}.

\item Under an additional assumption relating the interaction kernel and the potential, we also show that the difference of trajectories of Problem P$_{\alpha,\ep}$ and Problem P$_{0,0}$ emanating from the same initial data, is {\em{explicitly}} controlled, in the topology of the perturbation problem, in terms of the perturbation parameters $\alpha$ and $\ep$ on compact time intervals $[0,T]$.

\end{itemize}

It seems that such results for {\em{nonlocal}} Cahn-Hilliard equations do not yet appear in the literature.
These results show that the perturbation Problem P$_{\alpha,\ep}$ may be viewed as a ``relaxation'' of the limit Problem P$_{0,0}$ in the sense that, for any Problem P$_{0,0}$, there is a Problem P$_{\alpha,\ep}$ that is {\em{close}} (made more precise in Section \ref{s:upper-sc}).

\section{Preliminaries}  \label{s:framework}

Now we detail some preliminaries that will be applied to both problems.
To begin, define the spaces $H:=L^2(\Omega)$ and $V:=H^1(\Omega)$ with norms denoted by, $\|\cdot\|$ and $\|\cdot\|_V$, respectively. 
Otherwise, we write the norm of the Banach space $X$ with $\|\cdot\|_X$.
The inner-product in $H$ is denoted by $(\cdot,\cdot)$.
Denote the dual space of $V$ by $V'$, and the dual paring in $V'\times V$ is denoted by $\langle\cdot,\cdot\rangle.$
For every $\psi\in V'$, we denote by $\langle \psi \rangle$ the average of $\psi$ over $\Omega$, that is, 
\[
\langle \psi \rangle := \frac{1}{|\Omega|}\langle\psi,1\rangle,
\]
where $|\Omega|$ is the Lebesgue measure of $\Omega.$
Throughout, we denote by $\hat\psi:=\psi-\langle\psi\rangle$ and for future reference, observe $\langle\hat\psi\rangle=\langle \psi-\langle\psi\rangle \rangle=0.$
We will refer to the following norm in $V'$, which is equivalent to the usual one,
\begin{align}
\|\psi\|^2_{V'} = \left\| A^{-1/2}_N(\psi - \langle \psi \rangle) \right\|^2 + \langle \psi \rangle^2.  \notag
\end{align}

Define the space $L^2_0(\Omega):=\{\phi\in L^2(\Omega):\langle \phi \rangle=0\}.$
Let $A_N=-\Delta:L^2_0(\Omega)\rightarrow L^2_0(\Omega)$ with domain $D(A_N)=\{\psi\in H^2(\Omega):\partial_n\psi=0 \ \mathrm{on} \ \Gamma\}$ denote the ``Neumann-Laplace'' operator.
Of course the operator $A_N$ generates a bounded analytic semigroup, denoted $e^{-A_Nt}$, and the operator is nonnegative and self-adjoint on $L^2(\Omega).$ 
Recall, the domain $D(A_N)$ is dense in $H^2(\Omega).$
Further, define $V_0:=\{\psi\in V:\langle \psi \rangle=0\}$, and $V_0':=\{\psi\in V':\langle \psi \rangle=0\}$.
Then $A_N:V\rightarrow V'$, $A_N\in\mathcal{L}(V,V')$, is defined by, for all $u,v\in V$,
\[
\langle A_N u,v \rangle = \int_\Omega \nabla u\cdot \nabla v dx.
\]
It is well known that the restriction $A_{N\mid V_0}$ maps $V_0$ to $V_0'$ isomorphically, and the inverse map $\mathcal{N}=A_N^{-1}:V_0'\rightarrow V_0,$ is defined by, for all $\psi\in V_0'$ and $f\in V_0$
\[
A_N\mathcal{N}\psi=\psi, \quad \mathcal{N}A_Nf=f.
\]
Additionally, these maps satisfy the relations, for all $u\in V_0$ and $v,w\in V_0',$
\begin{align}
\langle A_N u,\mathcal{N}v\rangle = \langle u,v \rangle,  \label{NLr-1} \\ 
\langle v,\mathcal{N}w \rangle = \langle w,\mathcal{N}v \rangle.  \notag
\end{align}
The Sobolev space $V$ is endowed with the norm,
\begin{equation}  \label{H1-norm}
\|\psi\|^2_{V}:= \|\nabla\psi\|^2 + \langle \psi \rangle^2.
\end{equation}
Denote by $\lambda_\Omega>0$ the constant in the Poincar\'{e}-Wirtinger inequality,
\begin{equation}  \label{Poincare}
\|\psi-\langle\psi\rangle\| \le \sqrt{\lambda_\Omega}\|\nabla\psi\|.
\end{equation}
Whence, for $c_\Omega:=\max\{\lambda_\Omega,1\}$, there holds, for all $\psi\in V,$
\begin{align}
\|\psi\|^2 & \le \lambda_\Omega\|\nabla\psi\|^2 + \langle\psi\rangle^2  \label{Poincare2} \\ 
& \le c_\Omega\|\psi\|^2_{V}.  \notag
\end{align}

For each $m\ge0$, $\alpha>0$, and $\ep>0$ define the following energy phase-space for Problem P$_{\alpha,\ep}$,
\begin{align}
\mathbb{H}^{\alpha,\ep}_m:=\{ \zeta=(\phi,\theta)^{tr} \in H\times H : |\langle \phi \rangle|, |\langle \theta \rangle| \le m \},  \notag 
\end{align}
which is Hilbert when endowed with the $\alpha,\ep$-dependent norm whose square is given by,
\begin{align}
\|\zeta\|^2_{\mathbb{H}^{\alpha,\ep}_m} & := \|\phi\|^2_{V'} + \alpha\|\phi\|^2 + \ep\|\theta\|^2.  \notag 
\end{align}
When we are concerned with the dynamical system associated with Problem P$_{\alpha,\ep}$, we will utilize the following metric space
\begin{align}
\mathcal{X}^{\alpha,\ep}_m := \left\{ \zeta=(\phi,\theta)^{tr}\in\mathbb{H}^{\alpha,\ep}_m : F(\phi)\in L^1(\Omega) \right\},  \notag
\end{align}
endowed with the metric
\begin{align}
d_{\mathcal{X}^{\alpha,\ep}_m}(\zeta_1,\zeta_2) := \|\zeta_1-\zeta_2\|_{\mathbb{H}^{\alpha,\ep}_m} + \left| \int_\Omega F(\phi_1)dx - \int_\Omega F(\phi_2)dx \right|^{1/2}.  \notag
\end{align}
We also define the more regular phase-space for Problem P$_{\alpha,\ep}$,
\begin{align}
\mathbb{V}^{\alpha,\ep}_m:=\{ \zeta=(\phi,\theta)^{tr} \in V\times V : |\langle \phi \rangle|, |\langle \theta \rangle| \le m \},  \notag 
\end{align}
with the norm whose square is given by, $\|\zeta\|^2_{\mathbb{V}^{\alpha,\ep}_m} := \|\phi\|^2 + \alpha\|\phi\|^2_V + \ep\|\theta\|^2_V.$
Naturally, for Problem P$_{0,0}$ we set $\mathbb{H}^{0,0}_m:=\{ \phi \in H : |\langle \phi \rangle| \le m \}$ with $\|\zeta\|_{\mathbb{H}^{0,0}_m} := \|\phi\|_{V'}.$
Also, $\mathbb{V}^{0,0}_m := \{ \phi \in V : |\langle \phi \rangle| \le m \}$ with $\|\zeta\|_{\mathbb{V}^{0,0}_m} := \|\phi\|.$

The following assumptions on $J$ and $F$ are based on \cite{Frigeri&Grasselli12,Gal&Grasselli14}:

\begin{description}

\item[(H1)] $J\in W^{1,1}(\mathbb{R}^3)$, $J(-x)=J(x)$, and $a(x):=\int_\Omega J(x-y) dy > 0$ a.e. in $\Omega$.

\item[(H2)] $F\in C^{2,1}_{loc}(\mathbb{R})$ and there exists $c_0>0$ such that, for all $s\in\mathbb{R},$
\begin{align}
F''(s) + \inf_{x\in\Omega}a(x) \ge c_0.  \notag
\end{align}

\item[(H3)] There exists $c_1>\frac{1}{2}\|J\|_{L^1(\mathbb{R}^3)}$ and $c_2\in\mathbb{R}$ such that, for all $s\in\mathbb{R},$ 
\begin{align}
F(s)\ge c_1s^2 - c_2.  \notag 
\end{align}

\item[(H4)] There exists $c_3>0$, $c_4\ge0,$ and $p\in(1,2]$ such that, for all $s\in\mathbb{R},$
\begin{align}
|F'(s)|^p \le c_3|F(s)| + c_4.  \notag 
\end{align}

\item[(H5)] There exist $c_5,c_6>0,$ and $q>0$ such that, for all $s\in\mathbb{R},$
\begin{align}
F''(s) + \inf_{x\in\Omega}a(x) \ge c_5|s|^{2q} - c_6.  \notag
\end{align}

\end{description}

Let us make some remarks and report some important consequences of these assumptions. 
From \cite[Remark 2]{CFG12}: assumption (H2) implies that the potential $F$ is a quadratic perturbation of a (strictly) convex function. 
Indeed, if we set $a^*:=\|a\|_{L^\infty(\Omega)}$, then $F$ can be represented as 
\begin{equation}  \label{convex}
F(s)=G(s)-\frac{a^*}{2}s^2,
\end{equation}
with $G\in C^2(\mathbb{R})$ being strictly convex, since $G''\ge c_0$.
With (H3), for each $m\ge0$ there are constants $c_7,c_{8},c_{9},c_{10}>0$ (with $c_{8}$ and $c_{9}$ depending on $m$ and $F$) such that,
\begin{align}
F(s)-c_7\le c_{8}(s-m)^2 + F'(s)(s-m),  \label{Fcons-1}
\end{align}
\begin{align}
\frac{1}{2}|F'(s)|(1+|s|) \le F'(s)(s-m) + c_{9},  \label{Fcons-2}
\end{align}
and
\begin{align}
|F(s)| - c_{10} \le |F'(s)|(1+|s|).  \label{Fcons-3}
\end{align}
The last inequality appears in \cite[page 8]{Gal&Miranville09}.
With the positivity condition (H3), it follows that, for all $s\in\mathbb{R},$
\begin{equation}  \label{Fcons-3.1}
|F'(s)| \le c_3|F(s)|+c_4.
\end{equation}

{\em{A word of notation:}} In many calculations, functional notation indicating dependence on the variable $t$ is dropped; for example, we will write $\psi$ in place of $\psi(t)$. 
Throughout the article, $C>0$ will denote a \emph{generic} constant, while $Q:\mathbb{R}_{+}^d\rightarrow \mathbb{R}_{+}$ will denote a \emph{generic} increasing function in each of the $d$ components. 
Unless explicitly stated, all of these generic terms will be independent of the parameters $\alpha,$ $\delta,$ $\ep$, $T,$ and $m$.
Constants due to the embeddings $V'\hookleftarrow H$, or $H\hookleftarrow V$, are denoted by $C_\Omega.$
Finally, throughout we will use the following abbreviations
\[
c_J:=\|J\|_{L^1(\Omega)} \quad \text{and} \quad d_J:=\|\nabla J\|_{L^1(\Omega)}.
\]

We now review Problem P$_{0,0}$.

\section{The model problems}  \label{s:models}

First we recall several results for Problem P$_{0,0}.$

\subsection{The limit Problem P$_{0,0}$}

When we examine the limit Problem P$_{0,0}$, observe that through the time rescaling $s\mapsto(1+\delta^2)t$, 
\[
\phi_s(s) = (1+\delta^2)\phi_t((1+\delta^2)t)=\partial_t[\phi((1+\delta^2)t)],
\]
we subsequently do not need to include the term $\delta$ appearing in \eqref{lim-1} in this preliminary discussion.
The term $\delta$ will need to appear later when we compare both Problem P$_{0,0}$ and Problem P$_{\alpha,\ep}$ on the same (compact) time interval.
All of the following results for Problem P$_{0,0}$ are namely due to \cite{CFG12,Frigeri&Grasselli12} and can be found in \cite[Section 2.1]{Gal&Grasselli14}.

\begin{definition}  \label{d:lim-ws}
For $T>0$ and $\phi_0\in H$ with $F(\phi_0)\in L^1(\Omega)$, we say that $\phi$ is a {\sc{weak solution}} of Problem P$_{0,0}$ on $[0,T]$ if $\phi$ satisfies 
\begin{align}
& \phi\in C([0,T];H) \cap L^2(0,T;V),  \notag \\
& \phi_t\in L^2(0,T;V'),  \notag \\ 
& \mu=a(x)\phi - J*\phi + F'(\phi) \in L^2(0,T;V).  \notag
\end{align}
In addition, upon setting, 
\begin{align}
\rho = \rho(x,\phi) := a(x)\phi + F'(\phi),  \notag
\end{align}
for every $\varphi\in V,$ there holds, for almost all $t\in(0,T),$
\begin{align}
\langle\phi_t,\varphi\rangle + (\nabla\rho,\nabla\varphi) - (\nabla J*\phi,\nabla\varphi) & = 0.  \notag
\end{align}
Also, there holds,
\begin{align}
\phi(0) = \phi_0. \label{initial-phi}
\end{align}
We say that $\phi$ is a {\sc{global weak solution}} of Problem P$_{0,0}$ if it is a weak solution on $[0,T]$, for any $T>0.$
The initial condition \eqref{initial-phi} holds in the $L^2$-sense; i.e., for every $\varphi\in V,$
\begin{align}
(\phi(0),\varphi)=(\phi_0,\varphi).  \label{L2-initial-phi}
\end{align}
\end{definition}

It is well-known that the average value of $\phi$ is conserved (cf. e.g. \cite[Section III.4.2]{Temam01}). 
Indeed, taking $\varphi=1$ in \eqref{ws-7} yields, $\frac{\partial}{\partial t}\int_\Omega \phi(x,t) dx = 0$ and we naturally recover the {\em{conservation of mass}}
\begin{align}
\langle \phi(t) \rangle & = \langle \phi_0 \rangle \quad \text{and} \quad \partial_t\langle \phi(t) \rangle = 0.  \label{con-mass}
\end{align}

\begin{theorem}  \label{t:lim-existence}
Assume (H1)-(H5) hold with $p\in(\frac{6}{5},2]$ and $q\ge\frac{1}{2}$. For any $\phi_0\in H$ with $F(\phi_0)\in L^1(\Omega)$, there exists a unique global weak solution $\phi$ to Problem P$_{0,0}$ in the sense of Definition \ref{d:lim-ws} satisfying the additional regularity, for any $T>0$, 
\begin{eqnarray}
\phi & \in & L^\infty(0,T;L^{2+2q}(\Omega)),  \notag \\
F(\phi) & \in & L^\infty(0,T;L^1(\Omega)).  \notag
\end{eqnarray}
Furthermore, setting 
\begin{align}
\mathcal{E}_0(t):=\frac{1}{4} \int_\Omega \int_\Omega J(x-y)\left( \phi(x,t)-\phi(y,t) \right)^2 dxdy + \int_\Omega F(\phi(x,t)) dx,  \label{lim-wk-1}
\end{align} 
the following energy equality holds, for all $\phi_0\in H$ with $F(\phi_0)\in L^1(\Omega)$, and $t\in[0,T],$ 
\begin{align}
& \mathcal{E}_0(t) + \int_0^t \|\nabla\mu(s)\|^2 ds = \mathcal{E}_0(0).  \label{lim-wk-2}
\end{align}
\end{theorem}

\begin{proof}
See \cite[Theorem 2.2]{Gal&Grasselli14}, which follows \cite[Corollary 1 and Proposition 5]{Frigeri&Grasselli12} and \cite[Theorem 1]{CFG12}. 
\end{proof}

At this point in the discussion we can formalize the semi-dynamical system generated by Problem P$_{0,0}$.

\begin{corollary}  \label{t:lim-Lip-semif}
Let the assumptions of Theorem \ref{t:lim-existence} be satisfied. 
We can define a strongly continuous semigroup (of solution operators) $S_{0,0}=(S_{0,0}(t))_{t\ge0}$, 
\begin{equation*}
S_{0,0}(t):\mathcal{X}^{0,0}_m\rightarrow \mathcal{X}^{0,0}_m
\end{equation*}
by setting, for all $t\geq 0,$ 
\begin{equation*}
S_{0,0}(t)\phi_0:=\phi(t)
\end{equation*}
where $\phi(t)$ is the unique global weak solution to Problem P$_{0,0}$.
\end{corollary}

We now cite the result showing that Problem P$_{0,0}$ admits a global attractor.
The result is due to \cite[Theorem 4]{Frigeri&Grasselli12} (see also \cite[Theorem 2.7]{Gal&Grasselli14}).

\begin{theorem}  \label{t:lim-global}
The semigroup $S_{0,0}=(S_{0,0}(t))_{t\ge0}$ admits a global attractor $\mathcal{A}^{0,0}$ in $\mathbb{H}^{0,0}_m$. 
The global attractor is invariant under the semiflow $S_{0,0}$ (both positively and negatively) and attracts all nonempty bounded subsets of $\mathbb{H}^{0,0}_m$; precisely,

\begin{description}

\item[1] for each $t\geq 0$, $S_{0,0}(t)\mathcal{A}^{0,0} = \mathcal{A}^{0,0}$, and 

\item[2] for every nonempty bounded subset $B$ of $\mathbb{H}^{0,0}_m$,
\[
\lim_{t\rightarrow\infty}{\rm{dist}}_{\mathbb{H}^{0,0}_m}(S_{0,0}(t)B,\mathcal{A}^{0,0}) := \lim_{t\rightarrow\infty}\sup_{\zeta\in B}\inf_{\xi\in\mathcal{A}^{0,0}}\|S_{0,0}(t)\zeta-\xi\|_{\mathbb{H}^{0,0}_m} = 0.
\]

\end{description}

\noindent Additionally, 

\begin{description}

\item[3] the global attractor is the unique maximal compact invariant subset in $\mathbb{H}^{0,0}_m$ given by 
\[
\mathcal{A}^{0,0} := \omega(\mathcal{B}^{0,0}_0) := \bigcap_{s\geq 0}{\overline{\bigcup_{t\geq s}S_{0,0}(t)\mathcal{B}^{0,0}_0}}^{\mathbb{H}^{0,0}_m}.
\]

\end{description}

\noindent Furthermore, 

\begin{description}

\item[4] the global attractor is connected,

\item[5] the global attractor is bounded in $\mathbb{V}^{0,0}_m$, and 

\item[6] the fractal dimension of $\mathcal{A}^{0,0}$ is finite; i.e., 
\begin{equation*} 
{\rm{dim}}_F(\mathcal{A}^{0,0},\mathbb{H}^{0,0}_m):=\limsup_{r\rightarrow 0}\frac{\ln \mu_{\mathbb{H}^{0,0}_m}(\mathcal{A}^{0,0},r)}{-\ln r}<\infty,
\end{equation*}
where $\mu_{\mathbb{H}^{0,0}_m}(\mathcal{A}^{0,0},r)$ denotes the minimum number of balls of radius $r$ from $\mathbb{H}^{0,0}_m$ required to cover $\mathcal{A}^{0,0}$.

\end{description} 

\end{theorem}

\begin{proof}
The first three claims are a direct result of the existence of an absorbing set and the compactness of $S_{0,0}$ on $\mathbb{H}^{0,0}_m$. 
The fourth claim follows because Knesser's property is satisfied (cf. \cite[Section 5]{Ball04}).
The fifth claim is due to \cite[Corollary 2.9]{Gal&Grasselli14}.
Finally, the last claim follows due to the existence of an exponential attractor whose basis of attraction is the whole phase space (cf. \cite[Theorem 2.8]{Gal&Grasselli14}).
\end{proof}

The following result will be useful in Section \ref{s:upper-sc}.
This follows from \cite[Lemma 2.12]{Gal&Grasselli14}.

\begin{lemma}  \label{t:lim-reg-1}
Let the assumptions of Theorem \ref{t:lim-existence} be satisfied and assume $\phi$ is a weak solution to Problem P$_{0,0}$.
There exists a positive monotonically increasing function $Q(m)$ such that for all $t\in[0,T],$
\begin{align}
\int_0^t \|\phi_t(s)\|^2 ds \le Q(m)T.  \label{tip-0}
\end{align}
\end{lemma}

\subsection{The relaxation Problem P$_{\alpha,\ep}$}

Now we recall some important results for Problem P$_{\alpha,\ep}.$
The omitted proofs may be found in \cite{Shomberg-n16}; otherwise, as certain details in the next section rely on the proofs of some results stated here, we report, for the reader's convenience, those more important proofs in the appendix. 

\begin{definition}  \label{d:ws}
For $T>0$, $\delta_0>0$, $\delta\in(0,\delta_0]$, $(\alpha,\ep)\in(0,1]\times(0,1]$, and $\zeta_0=(\phi_0,\theta_0)^{tr}\in H\times H$ with $F(\phi_0)\in L^1(\Omega)$, we say that $\zeta=(\phi,\theta)^{tr}$ is a {\sc{weak solution}} of Problem P$_{\alpha,\ep}$ on $[0,T]$ if $\zeta=(\phi,\theta)^{tr}$ satisfies 
\begin{align}
& \phi\in C([0,T];H) \cap L^2(0,T;V),  \label{ws-1} \\
& \phi_t\in L^2(0,T;V'),  \label{ws-2} \\ 
& \sqrt{\alpha}\phi_t\in L^2(0,T;V),  \label{ws-2.1} \\ 
& \mu=a(x)\phi - J*\phi + F'(\phi) + \alpha\phi_t - \delta\theta \in L^2(0,T;V),  \label{ws-3} \\
& \theta\in C([0,T];H) \cap L^2(0,T;V),  \label{ws-4} \\
& \theta_t\in L^2(0,T;V').  \label{ws-5} 
\end{align}
In addition, upon setting, 
\begin{align}
\rho = \rho(x,\phi) := a(x)\phi + F'(\phi),  \label{ws-6}
\end{align}
for every $\varphi,\vartheta\in V,$ there holds, for almost all $t\in(0,T),$
\begin{align}
\langle\phi_t,\varphi\rangle + (\nabla\rho,\nabla\varphi) - (\nabla (J*\phi),\nabla\varphi) + \alpha(\nabla\phi_t,\nabla\varphi) & = \delta(\nabla\theta,\nabla\varphi)  \label{ws-7} \\ 
\langle\theta_t,\vartheta\rangle + (\nabla \theta,\nabla\vartheta) & = -\delta\langle\phi_t,\vartheta\rangle.  \label{ws-8}
\end{align}
Also, there holds,
\begin{align}
\phi(0) = \phi_0 \quad \text{and} \quad \theta(0) & = \theta_0.  \label{ws-9}
\end{align}
We say that $\zeta=(\phi,\theta)^{tr}$ is a {\sc{global weak solution}} of Problem P$_{\alpha,\ep}$ if it is a weak solution on $[0,T]$, for any $T>0.$
The initial conditions \eqref{ws-9} hold in the $L^2$-sense; i.e., for every $\vartheta\in V,$ \eqref{L2-initial-phi} and 
\begin{align}
(\theta(0),\vartheta)=(\theta_0,\vartheta)  \label{L2-initial-theta}
\end{align}
hold.
\end{definition}

In addition to \eqref{con-mass} and \eqref{con-mass}, taking $\vartheta=1$ in \eqref{ws-8} yields, $\frac{\partial}{\partial t}\int_\Omega \theta(x,t) dx = 0$ and we also establish 
\begin{align}
\langle \theta(t) \rangle & = \langle \theta_0 \rangle \quad \text{as well as} \quad \partial_t\langle \phi(t) \rangle = \partial_t\langle \theta(t) \rangle = 0.  \label{con-heat}
\end{align}
Together, \eqref{con-mass} and \eqref{con-heat} constitute {\em{conservation of enthalpy}}.

\begin{theorem}  \label{t:existence}
Assume (H1)-(H5) hold with $p\in(\frac{6}{5},2]$ and $q\ge\frac{1}{2}$. 
For any $\zeta_0=(\phi_0,\theta_0)^{tr}\in H\times H$ with $F(\phi_0)\in L^1(\Omega)$, there exists a global weak solution $\zeta=(\phi,\theta)^{tr}$ to Problem P$_{\alpha,\ep}$ in the sense of Definition \ref{d:ws} satisfying the additional regularity, for any $T>0$, 
\begin{eqnarray}
\phi & \in & L^\infty(0,T;L^{2+2q}(\Omega)),  \label{wk-0.001} \\
\sqrt{\alpha}\phi & \in & L^\infty(0,T;V),  \label{wk-0.01} \\
F(\phi) & \in & L^\infty(0,T;L^1(\Omega)),  \label{wk-0.2} \\ 
\theta_t & \in & L^2(0,T;H).  \label{wk-0.3}
\end{eqnarray}
Furthermore, setting 
\begin{align}
\mathcal{E}_\ep(t):=\frac{1}{4} \int_\Omega \int_\Omega J(x-y)\left( \phi(x,t)-\phi(y,t) \right)^2 dxdy + \int_\Omega F(\phi(x,t)) dx + \frac{\ep}{2}\int_\Omega \theta(t)^2 dx,  \label{wk-1}
\end{align} 
the following energy equality holds, for all $\zeta_0=(\phi_0,\theta_0)^{tr}\in \mathbb{H}^{\alpha,\ep}_m$ with $F(\phi_0)\in L^1(\Omega)$, and $t\in[0,T],$ 
\begin{align}
& \mathcal{E}_\ep(t) + \int_0^t \left( \|\nabla\mu(s)\|^2 + \alpha\|\phi_t(s)\|^2 + \|\nabla\theta(s)\|^2 \right) ds = \mathcal{E}_\ep(0).  \label{wk-2}
\end{align}
\end{theorem}

The following proposition establishes the uniqueness of weak solutions to Problem P$_{\alpha,\ep}$. 
Furthermore, it shows that the semigroup $S_{\alpha,\ep}$ (defined below) is strongly continuous with respect to the metric $\mathcal{X}^{\alpha,\ep}_m$.

\begin{proposition}  \label{t:cont-dep}
Assume (H1)-(H4) hold. 
Let $T>0$, $m\ge0$, $\delta_0>0$, $\delta\in(0,\delta_0]$, $(\alpha,\ep)\in(0,1]\times(0,1]$, and $\zeta_{01}=(\phi_{01},\theta_{01})^{tr}$, $\zeta_{02}=(\phi_{02},\theta_{02})^{tr}\in\mathbb{H}^{\alpha,\ep}_m$ be such that $F(\phi_{01}),F(\phi_{02})\in L^1(\Omega)$.
Let $\zeta_1(t)=(\phi_1(t),\theta_1(t))$ and $\zeta_2(t)=(\phi_2(t),\theta_2(t))$ denote the weak solution to Problem P$_{\alpha,\ep}$ corresponding to the data $\zeta_{01}$ and $\zeta_{02}$, respectively.
Then there are positive constants $\bar\nu_1=\bar\nu_1(c_0,J,\alpha,\ep,\delta_0)\sim\{\alpha^{-2},\ep^{-1}\}$ and $\bar\nu_2=\bar\nu_2(F,J,\Omega,\delta_0)$, independent of $T$, $\zeta_{01}$, and $\zeta_{02}$, such that, for all $t\in[0,T],$ 
\begin{align}
& \|\zeta_{1}(t)-\zeta_{2}(t)\|^2_{\mathbb{H}^{\alpha,\ep}_m} + \int_0^t \left( 2\|\partial_t\phi_1(s)-\partial_t\phi_2(s)\|^2_{V'} + \alpha\|\partial_t\phi_1(s)-\partial_t\phi_2(s)\|^2 + 2\|\theta_1(s)-\theta_2(s)\|^2_{V} \right) ds  \notag \\ 
& \le e^{\bar\nu_1 t} \left( \|\zeta_1(0)-\zeta_2(0)\|^2_{\mathbb{H}^{\alpha,\ep}_m} + \frac{2\bar\nu_2}{\bar\nu_1} \left( |M_1-M_2|+|N_1-N_2| \right)^2  \right)  \label{diff-0}
\end{align}
where $M_i:=\langle\phi_i(0)\rangle$, $N_i:=\langle\theta_i(0)\rangle$, $i=1,2$.
\end{proposition}

As before, we can now formalize the semi-dynamical system generated by Problem P$_{\alpha,\ep}$.

\begin{corollary}  \label{t:Lip-semif}
Let the assumptions of Theorem \ref{t:existence} be satisfied. 
We can define a strongly continuous semigroup (of solution operators) $S_{\alpha,\ep}=(S_{\alpha,\ep}(t))_{t\ge0}$, for each $\alpha>0$ and $\ep>0$,
\begin{equation*}
S_{\alpha,\ep}(t):\mathcal{X}^{\alpha,\ep}_m\rightarrow \mathcal{X}^{\alpha,\ep}_m
\end{equation*}
by setting, for all $t\geq 0,$ 
\begin{equation*}
S_{\alpha,\ep}(t)\zeta_0:=\zeta(t)
\end{equation*}
where $\zeta(t)=(\phi(t),\theta(t))$ is the unique global weak solution to Problem P$_{\alpha,\ep}$.
Furthermore, as a consequence of \eqref{diff-0}, if we assume 
\[
M_1=M_2 \quad\text{and}\quad N_1=N_2,
\]
the semigroup $S_{\alpha,\ep}(t):\mathcal{X}^{\alpha,\ep}_m\rightarrow \mathcal{X}^{\alpha,\ep}_m$ is Lipschitz continuous on $\mathcal{X}^{\alpha,\ep}_m$, uniformly in $t$ on compact intervals. 
\end{corollary}

We now give a dissipation estimate for Problem P$_{\alpha,\ep}$ from which we deduce the existence of an absorbing set.
The idea of the estimate follows \cite[Proposition 2]{Gal&Miranville09}.
It is here where we require the slight modification of hypothesis (H1).

\begin{lemma}  \label{t:diss-1}
Assume (H1)-(H4) hold.
Let $m\ge0$, $\delta_0>0$, $\delta\in(0,\delta_0]$, $(\alpha,\ep)\in(0,1]\times(0,1]$, $\zeta_0=(\phi_0,\theta_0)^{tr}\in \mathbb{H}^{\alpha,\ep}_m$ with $F(\phi_0)\in L^1(\Omega).$
Assume $\zeta=(\phi,\theta)^{tr}$ is a weak solution to Problem P$_{\alpha,\ep}$.
There is a positive constant $\nu_3=\nu_3(\delta_0,J,\Omega)$, but independent of $\alpha$, $\ep$, and $\zeta_0$, such that, for all $t\ge0$, the following holds, 
\begin{align}
& \|\hat\phi(t)\|^2_{V'} + \alpha\|\hat\phi(t)\|^2 + \|\sqrt{a}\phi(t)\|^2 + \|\hat\theta(t)\|^2 + (F(\phi(t)),1) - (J*\phi(t),\hat\phi(t))  \notag \\ 
& + \int_t^{t+1} \left( \|\phi_t(s)\|^2_{V'} + \alpha\|\phi_t(s)\|^2 + \|\theta(s)\|^2_{V} \right) ds  \notag \\
& \le Q(\|\zeta_0\|_{\mathbb{H}^{\alpha,\ep}_m})e^{-\nu_3t} + \frac{1}{\nu_3}Q(m),  \label{ap-1}
\end{align}
for some monotonically increasing functions $Q$.

Consequently, the set given by
\begin{equation}  \label{abs-set}
\mathcal{B}^{\alpha,\ep}_0 := \left\{ \zeta\in\mathbb{H}^{\alpha,\ep}_m : \|\zeta\|^2_{\mathbb{H}^{\alpha,\ep}_m} \le \frac{1}{\nu_3}Q(m)+1 \right\},
\end{equation}
where $Q(\cdot,\cdot)$ is the function from \eqref{ap-1}, is a closed, bounded absorbing set in $\mathbb{H}^{\alpha,\ep}_m$, positively invariant under the semigroup $S_{\alpha,\ep}$.
\end{lemma}

\begin{remark}  \label{r:uniform-radius}
According to the proof (see Appendix \ref{s:append}), $\nu_3$ is a function of $\delta_0$ and the relation is $\nu_3\sim 1-c\delta^2_0>0$ for a sufficiently small constant $c>0.$
\end{remark}

\begin{remark}
The following global uniform bound follows immediately from estimate \eqref{ap-1} and \eqref{ap-5}.
Under the assumptions of Lemma \ref{t:diss-1}, there holds
\begin{equation}  \label{unif-bnd}
\limsup_{t\rightarrow+\infty} \|\zeta(t)\|_{\mathbb{H}^{\alpha,\ep}_m} \le E(0) + \frac{1}{\nu_3}Q(m) =: Q(\|\zeta_0\|_{\mathbb{H}^{\alpha,\ep}_m},m)
\end{equation}
for a monotonically increasing function $Q$, independent of $\alpha$ and $\ep$
\end{remark}

With the existence of a bounded absorbing set set $\mathcal{B}_0^{\alpha,\ep}$ (in Lemma \ref{t:diss-1}), the existence of a global attractor now depends on the precompactness of the semigroup of solution operators $S_{\alpha,\ep}$. 
To this end, we know that there is a $t_*>0$ such that the map $S_{\alpha,\ep}(t_*)$ is a strict contraction on $\mathbb{H}^{\alpha,\ep}_m$, up to a precompact pseudometric $M_*$ (the proof is based on the proof of Proposition \ref{t:cont-dep}).
Such a contraction is commonly used in connection with phase-field type equations as an alternative to establish the precompactness of a semigroup; for particular recent results, also see, for example, \cite{Grasselli-2012,Grasselli-Schimperna-2011,Zheng&Milani05}.)
The existence of a global attractor in $\mathbb{H}^{\alpha,\ep}_m$ now follows by well-known arguments and can be found in \cite{Temam88,Babin&Vishik92} for example.
Additional characteristics of the attractor follow thanks to the gradient structure of Problem P$_{\alpha,\ep}$.
Indeed, from \eqref{wk-2} we see that if there is a $t_0>0$ in which 
\[
\mathcal{E}_\ep(t_0)=\mathcal{E}_\ep(0),
\]
then, for all $t\in(0,t_0),$ 
\begin{equation}  \label{no-grad}
\int_0^{t} \left( \|\nabla\mu(s)\|^2 + \alpha\|\phi_t(s)\|^2 + \|\nabla\theta(s)\|^2 \right)ds = 0.
\end{equation}
Hence, we deduce $\phi_t(t)=0$ and $\theta_t(t)=0$ for all $t\in(0,t_0)$.
Therefore, $\zeta=(\phi,\theta)^{tr}$ is a fixed point of the trajectory $\zeta(t)=S_{\alpha,\ep}(t)\zeta_0$.
Since the semigroup $S_{\alpha,\ep}(t)$ is precompact, the system $(\mathcal{X}^{\alpha,\ep}_m,S_{\alpha,\ep},\mathcal{E}_\ep)$ is gradient/conservative for each $\alpha\in(0,1]$ and $\ep\in(0,1]$.
In particular, the first three claims in the statement of the following theorem are a direct result of the existence of the an absorbing set, a Lyapunov functional $\mathcal{E}_\ep$, and the fact that the system $(\mathcal{X}^{\alpha,\ep}_m,S_{\alpha,\ep}(t),\mathcal{E}_\ep)$ is gradient. 
The fourth property is a direct result \cite[Theorem VII.4.1]{Temam88}, and the fifth follows from \cite[Theorem 6.3.2]{Zheng04}.

\begin{theorem}  \label{t:global}
For each $\alpha\in(0,1]$ and $\ep\in(0,1]$ the semigroup $S_{\alpha,\ep}=(S_{\alpha,\ep}(t))_{t\ge0}$ admits a global attractor $\mathcal{A}^{\alpha,\ep}$ in $\mathbb{H}^{\alpha,\ep}_m$. 
The global attractor is invariant under the semiflow $S_{\alpha,\ep}$ (both positively and negatively) and attracts all nonempty bounded subsets of $\mathbb{H}^{\alpha,\ep}_m$; precisely, 

\begin{description}

\item[1] for each $t\geq 0$, $S_{\alpha,\ep}(t)\mathcal{A}^{\alpha,\ep} = \mathcal{A}^{\alpha,\ep}$, and 

\item[2] for every nonempty bounded subset $B$ of $\mathbb{H}^{\alpha,\ep}_m$,
\[
\lim_{t\rightarrow\infty}{\rm{dist}}_{\mathbb{H}^{\alpha,\ep}_m}(S_{\alpha,\ep}(t)B,\mathcal{A}^{\alpha,\ep}) := \lim_{t\rightarrow\infty}\sup_{\zeta\in B}\inf_{\xi\in\mathcal{A}^{\alpha,\ep}}\|S_{\alpha,\ep}(t)\zeta-\xi\|_{\mathbb{H}^{\alpha,\ep}_m} = 0.
\]

\end{description}

\noindent \noindent Additionally,

\begin{description}

\item[3] the global attractor is unique maximal compact invariant subset in $\mathbb{H}^{\alpha,\ep}_m$ given by
\[
\mathcal{A}^{\alpha,\ep} := \omega (\mathcal{B}^{\alpha,\ep}_0) := \bigcap_{s\geq 0}{\overline{\bigcup_{t\geq s}S_{\alpha,\ep}(t)\mathcal{B}^{\alpha,\ep}_0}}^{\mathbb{H}^{\alpha,\ep}_m}.
\]

\end{description}

\noindent Furthermore, 

\begin{description}

\item[4] the global attractor $\mathcal{A}^{\alpha,\ep}$ is connected and given by the union of the unstable manifolds connecting the equilibria of $S_{\alpha,\ep}(t)$, and

\item[5] for each $\zeta_0=(\phi_0,\theta_0)^{tr}\in\mathbb{H}^{\alpha,\ep}_m$, the set $\omega(\zeta_0)$ is a connected compact invariant set, consisting of the fixed points of $S_{\alpha,\ep}(t).$

\end{description}

\end{theorem}

The next result shows the global attractor is bounded in a more regular space.
Further regularity results can be found in \cite[Section 4.4]{Shomberg-n16}.

\begin{lemma}
Under the assumptions of Lemma \ref{t:diss-1}, the set given by
\begin{equation}  \label{reg-set}
\mathcal{B}^{\alpha,\ep}_1 := \left\{ \zeta\in\mathbb{V}^{\alpha,\ep}_m : \|\zeta\|^2_{\mathbb{V}^{\alpha,\ep}_m} \le \left( \frac{1}{\ep}+1 \right) \left( E(0) + \left( \frac{2}{\nu_3}+1 \right) Q_\alpha(m) +1 \right) \right\},
\end{equation}
for some positive monotonically increasing function $Q_\alpha\sim\alpha^{-1}$, is a closed, bounded absorbing set in $\mathbb{V}^{\alpha,\ep}_m$, positively invariant under the semigroup $S_{\alpha,\ep}$.

Furthermore, each global attractor $\mathcal{A}^{\alpha,\ep}$ is bounded in $\mathbb{V}^{\alpha,\ep}_m$, i.e., $\mathcal{A}^{\alpha,\ep}\subset\mathcal{B}^{\alpha,\ep}_1,$ and compact in $\mathbb{H}^{\alpha,\ep}_m$.
\end{lemma}

\begin{remark}
The ``radius'' of the set $\mathcal{B}^{\alpha,\ep}_1$ depends on $\alpha$ and $\ep$ like, respectively, $\alpha^{-1}$ and $\ep^{-1}$.
\end{remark}

The final result is this section concerns bounding the global attractor $\mathcal{A}^{\alpha,\ep}$ in a more regular subspace of $\mathbb{V}^{\alpha,\ep}_m$.
For each $m\ge0$, $\alpha\in(0,1]$ and $\ep\in(0,1]$, we now define the regularized phase-space
\begin{align}
\mathbb{W}^{\alpha,\ep}_m := \{ \zeta=(\phi,\theta)^{tr} \in \mathbb{V}^{\alpha,\ep}_m : \sqrt{\alpha}\mu\in H^2(\Omega),\ |\langle \phi \rangle|, |\langle \theta \rangle| \le m \},  \notag 
\end{align}
with the norm inherited from $\mathbb{V}^{\alpha,\ep}_m$.
Also, we define the following metric space
\begin{align}
\mathcal{Y}^{\alpha,\ep}_m := \left\{ \zeta=(\phi,\theta)^{tr}\in\mathbb{W}^{\alpha,\ep}_m : F(\phi)\in L^1(\Omega) \right\},  \notag
\end{align}
endowed with the metric 
\begin{align}
d_{\mathcal{Y}^{\alpha,\ep}_m}(\zeta_1,\zeta_2) := \|\zeta_1-\zeta_2\|_{\mathbb{V}^{\alpha,\ep}_m} + \left| \int_\Omega F(\phi_1)dx - \int_\Omega F(\phi_2)dx \right|^{1/2}.  \notag
\end{align}

\begin{theorem}  \label{t:reg-set}
For each $\alpha\in(0,1]$, $\ep\in(0,1]$ and for any $t\ge t_*$, the semigroup $S_{\alpha,\ep}$ satisfies $S_{\alpha,\ep}(t):\mathcal{X}^{\alpha,\ep}_m\rightarrow\mathcal{Y}^\alpha_m.$
Moreover, the global attractor $\mathcal{A}^{\alpha,\ep}$ admitted by the semigroup $S_{\alpha,\ep}$ is bounded in $\mathbb{W}^{\alpha,\ep}_m$ and compact in $\mathbb{H}^{\alpha,\ep}_m.$
\end{theorem}

\section{Upper-semicontinuity of the global attractors}  \label{s:upper-sc}

The following semicontinuity results are not possible until we provide a natural embedding for the attractor of Problem P$_{0,0}$ (here called a {\em{lift}}) into the phase-space of Problem P$_{\alpha,\ep}$.
To add motivation for the definition of the lift map, it is first worthwhile to notice that the equations \eqref{lim-1}-\eqref{lim-6} are equivalent to the system 
\begin{eqnarray}
\phi^0_t = \Delta\mu^0 &\text{in}& \Omega\times(0,T)  \label{eqlim-1} \\ 
\mu^0 = a\phi^0 - J*\phi^0 + F'(\phi^0) - \delta\theta^0 &\text{in}& \Omega\times(0,T)  \label{eqlim-2} \\ 
-\Delta\theta^0 = -\delta\phi^0_t &\text{in}& \Omega\times(0,T)  \label{eqlim-3} \\
\partial_n\mu^0 = 0 &\text{on}& \Gamma\times(0,T)  \label{eqlim-4} \\ 
\phi^0(x,0) = \phi_0(x) &\text{at}& \Omega\times\{0\}.  \label{eqlim-6}
\end{eqnarray}
Hence, in our setting, a {\em{lift}} is a mapping $\mathcal{L}:\mathbb{H}^{0,0}_m\rightarrow\mathbb{H}^{\alpha,\ep}_m$ defined by
\[
\mathcal{L}(\phi^0):=(\phi^0,\theta^0=\mathcal{M}(\phi^0)),
\]
where $\mathcal{M}:H\rightarrow H$ is a so-called {\em{canonical extension}} operator, given by
\begin{equation}  \label{extension}
\mathcal{M}(\phi^0):=\delta\mu^0.
\end{equation}
It should be noted that the chemical potential $\mu^0$, and hence $\mathcal{M}$, regularizes into $V$ for any $t>0$; more precisely, we cite \cite[Proof of Lemma 2.17]{Gal&Grasselli14}, where, for any $\tau>0$
\[
\sup_{t\ge\tau}\|\mu^0(t)\|_V \le Q(m,\tau).
\]
The {\em{lifted limit}} Problem P$_{0,0}$ can then be described by
\begin{eqnarray}
\phi^0_t = \Delta\mu^0 &\text{in}& \Omega\times(0,T)  \label{lif-1} \\ 
\mu^0 = a\phi^0 - J*\phi^0 + F'(\phi^0) - \delta\theta^0 &\text{in}& \Omega\times(0,T)  \label{lif-2} \\ 
-\Delta\theta^0 = -\delta \phi_t^0 &\text{in}& \Omega\times(0,T)  \label{lif-3} \\ 
\partial_n\mu^0 = 0 &\text{on}& \Gamma\times(0,T)  \label{lif-4} \\ 
\partial_n\theta^0 = 0 &\text{on}& \Gamma\times(0,T)  \label{lif-5} \\ 
\phi^0(x,0) = \phi_0(x) &\text{at}& \Omega\times\{0\}  \label{lif-6} \\ 
\theta^0(x,0) = \mathcal{M}(\phi_0(x)) &\text{at}& \Omega\times\{0\}. \label{lif-7}  
\end{eqnarray}
We emphasize that the {\em{lift}} of the initial data $\phi_0\in\mathbb{H}^{0,0}_m$ from Problem P$_{0,0}$ is given by $\mathcal{L}(\phi_0) = (\phi_0,\mathcal{M}(\phi_0))$.

\begin{remark}
Notice that in the original formulation of the limit problem (see \eqref{lim-1}-\eqref{lim-6}) we need the term $1+\delta^2$ in equation \eqref{lim-1} because later (below) we want to compare the difference between the perturbation problem and the lifted limit problem on the {\em{same}} compact time interval.
This observation is important because we will later rescale the time variable in order to obtain a suitable control of the problems in the weak energy phase space.
\end{remark}

For each $(\alpha,\ep)\in[0,1]\times[0,1]$, define the family of sets  in $\mathbb{H}^{\alpha,\ep}_m,$
\begin{equation}  \label{family}
\mathbb{A}^{\alpha,\ep} :=
\left\{ \begin{array}{ll} \mathcal{A}^{\alpha,\ep} & \text{for}\quad (\alpha,\ep)\in(0,1]\times(0,1] \\ \mathcal{L}\mathcal{A}^{0,0} & \text{when}\quad \alpha=0 \quad\text{and}\quad \ep=0. \end{array} \right.
\end{equation}

\subsection{A ``classical'' upper-semicontinuity result}

The main result of this section immediately follows.

\begin{theorem}  \label{t:upper-continuity} 
The family $\{\mathbb{A}^{\alpha,\varepsilon}\}_{\alpha,\varepsilon \in (0,1]},$ defined by (\ref{family}), is upper-semicontinuous at $\alpha=0$ and $\varepsilon=0$ in the metric space $\mathcal{X}^{1,1}_{m}$. 
More precisely, there holds 
\begin{equation}
\lim_{\alpha,\varepsilon \rightarrow 0}\mathrm{dist}_{\mathcal{X}^{1,1}_{m}}(\mathbb{A}^{\alpha,\varepsilon},\mathbb{A}^{0,0}):=\lim_{\alpha,\varepsilon \rightarrow 0}\sup_{a\in \mathcal{A}^{\alpha,\varepsilon}}\inf_{b\in \mathcal{LA}^{0,0}}\|a-b\|_{\mathcal{X}^{1,1}_{m}}=0.  \label{eq:upper-semicontinuity}
\end{equation}
\end{theorem}

\begin{proof}
The proof essentially follows the classical argument in \cite{Hale&Raugel88,Hale88} and also \cite{Zheng&Milani05}. 
Of course several modifications are required to account for the precise model problems considered here.
Let $\zeta=(\phi,\theta)\in \mathcal{A}^{\alpha,\varepsilon}$
and $\bar{\zeta}=(\bar{\phi},\bar{\theta})\in \mathcal{LA}^{0,0}$. 
We need to show that, as $\alpha,\varepsilon\rightarrow 0$,
\begin{align}
\sup_{(\phi,\theta)\in\mathcal{A}^{\alpha,\varepsilon}}\inf_{(\bar\phi,\bar\theta)\in\mathcal{LA}^{0,0}} & \left( \|\phi-\bar\phi\|^2_{V'} + \|\phi-\bar\phi\|^2 + \|\theta-\bar\theta\|^2 + \left| \int_\Omega F(\phi) dx - \int_\Omega F(\bar\phi) dx \right| \right)^{1/2} \rightarrow 0.  \label{eq:upper-1}
\end{align}
Assuming to the contrary that (\ref{eq:upper-1}) did not hold, then there
exist $\eta_{0}>0$ and sequences $(\alpha_{n})_{n\in \mathbb{N}}\subset (0,1]$, $(\varepsilon_{n})_{n\in \mathbb{N}}\subset (0,1]$, $(\zeta_{n})_{n\in \mathbb{N}}=((\phi_{n},\theta_{n}))_{n\in \mathbb{N}}\subset \mathcal{A}^{\alpha_n,\varepsilon_n}$, such that $\alpha_n\rightarrow 0$, $\varepsilon_{n}\rightarrow 0$, and for all $n\in \mathbb{N}$, 
\begin{equation}
\inf_{(\bar{\phi},\bar{\theta})\in \mathcal{LA}^{0,0}} \left( \|\phi_{n}-\bar{\phi}\|_{V'}^{2} + \|\phi_{n}-\bar{\phi}\|^{2} + \|\theta_{n}-\bar{\theta}\|^{2} + \left| \int_\Omega F(\phi_n) dx - \int_\Omega F(\bar\phi) dx \right| \right) \ge \eta_{0}^{2}.  \label{eq:upper-2}
\end{equation}
By Theorem \ref{t:reg-set}, the compact sets $\mathcal{A}^{\alpha_n,\varepsilon_n}$ are bounded in the space $\mathbb{W}^{1,1}_{m}$ and we have
the following uniform bound, for some positive constant $C>0$ independent of 
$n$, 
\begin{equation*}
\|\phi_{n}\|_{V'}^{2} + \|\phi_{n}\|^{2} + \|\theta_{n}\|^{2} \le C.
\end{equation*}
This means that there is a weakly converging subsequence of $(\zeta
_{n})_{n\in \mathbb{N}}$ (not relabelled) that converges to some $(\phi^*,\theta^*)$ weakly in $\mathbb{W}^{1,1}_{m}$. 
By the compactness of the embedding $\mathbb{W}^{1,1}_{m}\hookrightarrow \mathbb{H}^{1,1}_{m}$, the subsequence converges strongly in $\mathbb{H}^{1,1}_{m}.$ 
It now suffices to show that $(\phi^*,\theta^*)\in \mathcal{LA}^{0,0},$ since this is
a contradiction to (\ref{eq:upper-2}).

With each $\zeta_{n}=(\phi_{n},\theta_{n})\in \mathcal{A}^{\alpha_n,\varepsilon_n}$, for $n\in \mathbb{N}$, there is a complete orbit 
\begin{equation*}
(\phi^{n}(t),\theta^{n}(t))_{t\in \mathbb{R}}=(S_{\alpha_n,\varepsilon_{n}}(t)(\phi_{n},\theta_{n}))_{t\in \mathbb{R}}
\end{equation*}
contained in $\mathcal{A}^{\alpha_n,\varepsilon_{n}}$ and passing through $(\phi_{n},\theta_{n})$ where 
\begin{equation*}
(\phi^{n}(0),\theta^{n}(0)) = (\phi_{n},\theta_{n})
\end{equation*}
(cf., e.g., \cite[Proposition 2.39]{Milani&Koksch05}).
In view of the regularity $\mathcal{A}^{\alpha_n,\varepsilon_n}\subset (\mathcal{B}^{1,1}_1 \cap \mathbb{W}^{1,1}_{m})$, we obtain the uniform bounds,
\begin{equation}
\|\phi^n(t)\|^2_{V} + \|\theta^n(t)\|^2_{V} + \alpha_n\|\mu^n(t)\|^2_{H^2(\Omega)} \le C,  \label{eq:upper-3}
\end{equation}
where the constant $C>0$ is independent of $t$, $\alpha_n$, and $\varepsilon _{n}$. 
Additionally, from \cite[inequality (3.170)]{Shomberg-n16} we also have the uniform bounds 
\begin{equation}
\|\phi^n_t(t)\|^2_{V'} \le C.  \label{eq:upper-4}
\end{equation}
Now, for all $T>0$, the functions $\phi^{n}$, $\theta^{n}$, $\sqrt{\alpha_n}\phi_t^{n}=\sqrt{\alpha_n}\Delta\mu^{n}$, and $\phi_t^n$ are, respectively, bounded in $L^{\infty }(-T,T;V)$, $L^{\infty}(-T,T;V)$, $L^{\infty }(-T,T;H)$ and $L^{\infty }(-T,T;V')$. 
Thus, there is a function $\zeta=(\phi,\theta)$ and a subsequence (not relabelled), in which, 
\begin{equation}
\phi^n \rightharpoonup \phi \quad \text{weakly-* in} \quad L^{\infty}(-T,T;V),  \label{eq:conv-1}
\end{equation}
\begin{equation}
\theta^n \rightharpoonup \theta \quad \text{weakly-* in} \quad L^{\infty}(-T,T;V),  \label{eq:conv-2}
\end{equation}
\begin{equation}
\sqrt{\alpha_n}\phi_t^n \rightharpoonup 0 \quad \text{weakly-* in} \quad L^{\infty}(-T,T;H),  \label{eq:conv-3}
\end{equation}
\begin{equation}
\phi_t^n \rightharpoonup \phi_t \quad \text{weakly-* in} \quad L^{\infty}(-T,T;V').  \label{eq:conv-4}
\end{equation}
By virtue of the compact embedding 
\begin{equation}
\{ \psi\in L^{\infty}(-T,T;V) : \psi_{t}\in L^{\infty}(-T,T;V') \}\hookrightarrow C([-T,T];H)  \label{eq:injection-1}
\end{equation}
(see, e.g., \cite{Lions69}) the above convergence properties yield
\begin{equation}
\phi^{n} \rightarrow \phi \quad \text{strongly in} \quad C([-T,T];H).  \label{eq:conv-6}
\end{equation}
Thus, directly from (H2),
\begin{align}
\sup_{t\in[-T,T]} \|F'(\phi^n(t))-F'(\phi(t))\|^2 & \le \sup_{t\in[-T,T]} C\|\phi^n(t)-\phi(t)\|^{2},  \notag
\end{align}
for some positive constant $C$ independent of $n$, $\alpha_n$, and $\varepsilon_{n}.$ 
By virtue of (\ref{eq:conv-6}) it is then easy to see that
\begin{equation*}
F'(\phi^n) \rightarrow F'(\phi) \quad \text{strongly in} \quad C([-T,T];H).
\end{equation*}
It follows that $(\phi,\theta)$ is a weak solution of the limit problem P$_{0,0}$ on $\mathbb{R}$ (see \eqref{lif-1}-\eqref{lif-7}). 
In particular, $(\phi_{n},\theta_{n})=(\phi^{n}(0),\theta^{n}(0))\rightarrow (\phi(0),\theta(0))$ in $\mathbb{V}^{0,0}_m$.
Therefore, we have that $(\phi(0),\theta(0))=(\phi^*,\theta^*)$ and $(\phi(0),\theta(0))\in \mathbb{V}^{0,0}_m$. 
As $(\phi,\theta)$ is a complete orbit through $(\phi^*,\theta^*)$, it follows that $(\phi^*,\theta^*)\in \mathcal{LA}^{0,0}$, in contradiction to (\ref{eq:upper-2}). 
This completes the proof.
\end{proof}

\begin{remark}
It is important to note that the above convergence result appears in the metric induced by the topology of $\mathbb{H}^{1,1}_m$; i.e., $\alpha$ and $\ep$ are {\em{fixed}} in the norm.
This is contrary to the result that follows in the next section.
\end{remark}

\subsection{An upper-semicontinuity type result; explicit control over semidistances}

This section is devoted to an upper-semicontinuity type result for the family of global attractors admitted by Problem P$_{\alpha,\ep}$ and Problem P$_{0,0}$.
A major step toward establishing this result is demonstrating the difference between trajectories of each problem, corresponding to the same initial data $\phi_0$ can be controlled, in the topology of $\mathbb{H}^{\alpha,\ep}_m$, explicitly in terms of the relaxation parameters $\alpha$ and $\ep$.
The result given here will depend on an important assumption that allows for needed control between the kernel $J$ and the potential $F.$

\begin{description}

\item[(H6)] The constants $c_J=\|J\|_{L^1(\Omega)}$ and $c_0>0$ of (H2) satisfy $c_0>c_J.$

\end{description}

\begin{remark}
By \eqref{convex}, for all $s\in\mathbb{R},$
\[
F''(s) \ge c_0-a^*.
\]
Hence, the condition in (H6) is satisfied for any $J$ in (H1) in which
\[
\inf_{x\in\Omega}\int_\Omega \left( J(y-x)-J(y) \right) dy > -(c_0-a^*).
\]
This means the double-well potential is admissible for suitable (``quasiconvex'') kernels $J$. 
\end{remark}

The following results will lead us to the upper-semicontinuity type result.
In large part, this is possible since the radius of the absorbing set $\mathcal{B}^{\alpha,\ep}_0$ (and hence the bound on $\mathcal{A}^{\alpha,\ep}$) is {\em{independent}} of $\alpha$, $\ep$ (see Remark \ref{r:uniform-radius}).

\begin{lemma}  \label{t:du}
Assume (H1)-(H4) hold.
Let $m\ge0$, $R>0$, $T>0$, $\delta_0>0$, and $\delta\in(0,\delta_0]$.
For all $\phi_0\in \mathbb{V}^{0,0}_m$ with $F(\phi_0)\in L^1(\Omega)$ and $\|\phi_0\|_{\mathbb{V}^{0,0}_m}\le R$.
Assume $\phi^0$ is a weak solution to Problem P$_{0,0}$.
There exists a positive monotonically increasing function $Q$, depending on $R$ and $m$, such that for all $t\in[0,T]$, there hold
\begin{align}
(1+\delta^2)\int_0^t \|\phi^0_{tt}(s)\|^2_{V'}ds \le Q(R,m)T \quad \text{and hence,} \quad \int_0^t \|\mu^0_{t}(s)\|^2_{V}ds \le Q(R,m)T.  \label{du-0}
\end{align}
\end{lemma}

\begin{proof}
Differentiate \eqref{lim-1} and \eqref{lim-2} with respect to $t$ and set $u^0=\phi^0_t$ and $m^0=\mu^0_t.$
The differentiated system is 
\begin{eqnarray}
(1+\delta^2)u^0_t = \Delta m^0 &\text{in}& \Omega\times(0,T)  \label{du-1} \\ 
m^0 = au^0 - J*u^0 + F''(\phi^0)u^0  &\text{in}& \Omega\times(0,T)  \label{du-2} \\ 
\partial_nm^0 = 0 &\text{on}& \Gamma\times(0,T)  \label{du-3} \\ 
u^0(x,0) = (1+\delta^2)^{-1}\Delta\mu^0(0) &\text{at}& \Omega\times\{0\}.  \label{du-4}
\end{eqnarray}
Multiply \eqref{du-1} by $A^{-1}_Nu_t^0$ in $L^2(\Omega)$ and multiply \eqref{du-2} by $u_t^0$ in $L^2(\Omega)$, and sum the results to obtain 
\begin{align}
\frac{1}{4}\frac{d}{dt} \left\{ \int_\Omega\int_\Omega J(x-y)(u^0(x)-u^0(y))^2 dxdy \right\} + (F''(\phi^0)u^0,u^0_t) + (1+\delta^2)\|u^0_t\|^2_{V'} = 0.  \label{du-5}
\end{align}
By \eqref{convex}, 
\[
(F''(\phi^0)u^0,u^0_t) \ge (c_0-a^*)(u^0,u^0_t) = (c_0-a^*)\frac{1}{2}\frac{d}{dt}\|u^0\|^2,
\]
and we find
\begin{align}
\frac{d}{dt} \left\{ \int_\Omega\int_\Omega J(x-y)(u^0(x)-u^0(y))^2 dxdy + 2(c_0-a^*)\|u^0\|^2 \right\} + 4(1+\delta^2)\|u^0_t\|^2_{V'} = 0.  \label{du-55}
\end{align}
Thus, integrating \eqref{du-55} over $(0,t)$ immediately shows \eqref{du-0}.
This finishes the proof.
\end{proof}

For the following we define the {\em{projection}} $\Pi:\mathbb{H}^{\alpha,\ep}_m\rightarrow\mathbb{H}^{0,0}_m$ given by $\Pi(\phi,\theta)=\phi.$
Also, we remind the reader that the constant due to the embedding $H\hookrightarrow V'$ is denoted by $C_\Omega>0$.

\begin{lemma}  \label{t:rob-1}
Assume (H1)-(H4) and (H6) hold.
Let $m\ge0$, $R>0$, $T>0$, $\delta_0>0$, $\delta\in(0,\delta_0]$.
There exists a positive monotonically increasing function $Q$, depending on $R$, $m$ and $T$, such that for all $t\in[0,T]$, $(\alpha,\ep)\in(0,1]\times(0,1]$, and for all $\zeta_0=(\phi_0,\theta_0)^{tr}\in \mathbb{V}^{\alpha,\ep}_m$ with $\|\zeta_0\|_{\mathbb{V}^{\alpha,\ep}_m}\le R$, there holds
\begin{align}
\|\tilde\phi(t)\|^2_{V'} + \alpha\|\tilde\phi(t)\|^2 + \ep\|\tilde\theta(t)\|^2 & + \int_0^t \left( \|\tilde\phi_t(s)\|^2_{V'} + \alpha\|\tilde\phi_t(s)\|^2 + \ep\|\tilde\theta(s)\|^2_V \right) ds  \notag \\ 
& \le (\alpha+\ep)^{1/2} Q(R,m,T).  \label{rob-1}
\end{align}
\end{lemma}

\begin{proof}
Then the difference $\tilde\zeta=(\tilde\phi,\tilde\theta)=(\phi^{+},\theta^{+})-(\phi^0,\theta^0)$, representing the difference between Problem P$_{\alpha,\ep}$ and the lifted Problem P$_{0,0}$ \eqref{lif-1}-\eqref{lif-7}, satisfies the equations
\begin{eqnarray}
\tilde\phi_t = \Delta\tilde\mu &\text{in}& \Omega\times(0,T)  \label{scl-1} \\ 
\tilde\mu = a\tilde\phi - J*\tilde\phi + F'(\phi^+) - F'(\phi^0) + \alpha\tilde\phi_t - \delta\tilde\theta + \alpha\phi^0_t &\text{in}& \Omega\times(0,T)  \label{scl-2} \\ 
\ep\tilde\theta_t - \Delta\tilde\theta = -\delta \tilde\phi_t - \ep\theta^0_t &\text{in}& \Omega\times(0,T)  \label{scl-3} \\ 
\partial_n\tilde\mu = 0 &\text{on}& \Gamma\times(0,T)  \label{scl-4} \\ 
\partial_n\tilde\theta = 0 &\text{on}& \Gamma\times(0,T)  \label{scl-5} \\ 
\tilde\phi(x,0) = 0 &\text{at}& \Omega\times\{0\}  \label{scl-6} \\ 
\tilde\theta(x,0) = \theta_0(x)-\mathcal{M}(\phi_0(x)) &\text{at}& \Omega\times\{0\}. \label{scl-7}  
\end{eqnarray}
Observe, 
\begin{align}
\langle\tilde\phi\rangle=0 \quad \text{and} \quad \langle\tilde\theta\rangle = \langle \theta_0 \rangle - \langle \mathcal{M}(\phi_0) \rangle.  \label{usc-e}
\end{align}

Multiply \eqref{scl-1}-\eqref{scl-3} in $L^2(\Omega)$ by $A_N^{-1}\tilde\phi_t$, $\tilde\phi_t$, and $\tilde\theta$, respectively, and sum the resulting identities to find, for all $t\in[0,T],$ 
\begin{align}
& \frac{d}{dt} \left\{ \ep\|\tilde\theta\|^2 \right\} + 2\|\tilde\phi_t\|^2_{V'} + 2\alpha\|\tilde\phi_t\|^2 + 2\|A^{1/2}_N\tilde\theta\|^2  \notag \\ 
& + 2(a\tilde\phi+F'(\phi_1)-F'(\phi_2),\tilde\phi_t) - 2(J*\tilde\phi,\tilde\phi_t)  \notag \\ 
& = 2|\Omega|\langle \tilde\phi \rangle \langle \tilde\mu \rangle - \alpha(\phi^0_t,\tilde\phi_t) - \ep(\theta^0_t,\tilde\theta).  \label{usc-18}
\end{align}
Estimating the resulting products using assumption (H2) yields, for any $\eta>0,$
\begin{align}
2(a\tilde\phi + F'(\phi_1) - F'(\phi_2),\tilde\phi_t) & \ge 2c_0(\tilde\phi,\tilde\phi_t) = c_0\frac{d}{dt}\|\tilde\phi\|^2,  \label{diff-11}
\end{align}
and
\begin{align}
-2(J*\tilde\phi,\tilde\phi_t) & = -\frac{d}{dt}(J*\tilde\phi,\tilde\phi).  \label{diff-13}
\end{align}
Let us now estimate the auxiliary terms 
\begin{align}
\alpha|(\phi^0_t,\tilde\phi_t)| & \le \frac{\alpha}{4}\|\phi^0_t\|^2 + \alpha\|\tilde\phi_t\|^2  \label{usc-18.1}
\end{align}
and
\begin{align}
\ep|(\theta^0_t,\tilde\theta)| \le \ep\|A^{-1/2}_N\theta^0_t\|^2 + \|A^{1/2}_N\tilde\theta\|^2.  \label{usc-19}
\end{align}
Combining \eqref{usc-18}-\eqref{usc-19} and recalling $\ep\in(0,1]$ and \eqref{H1-norm}, we have, for almost all $t\in[0,T],$
\begin{align}
& \frac{d}{dt} \left\{ \ep\|\tilde\theta\|^2 + c_0\|\tilde\phi\|^2 - (J*\tilde\phi,\tilde\phi) \right\} + 2\|\tilde\phi_t\|^2_{V'} + \alpha\|\tilde\phi_t\|^2 + \ep\|\tilde\theta\|^2_V  \notag \\ 
& \le C|\Omega||\langle \tilde\phi \rangle| |\langle \tilde\mu \rangle| + \ep\langle\tilde\theta\rangle^2 + \frac{\alpha}{4}\|\phi^0_t\|^2 + \ep\|A^{-1/2}_N\theta^0_t\|^2.  \label{usc-19.1}
\end{align}
Using the local Lipschitz assumption (H2), it is easy to show that, 
\begin{align}
|\langle \tilde\mu \rangle| & \le C_F|\langle \tilde\phi \rangle| + \delta_0|\langle \tilde\theta \rangle|  \notag \\ 
& =:\tilde\mu_*,  \label{diff-19.1}
\end{align}
for some positive constant $C_F$ depending on $c_J$ and the Lipschitz bound on $F'$.
Thanks to \eqref{tip-0} we have 
\begin{align}
\int_0^t \alpha\|\phi^0_t(s)\|^2 ds \le \alpha Q(m)T,  \label{usc-20.1}
\end{align}
and we also have the following thanks to \eqref{du-0}$_1$ (recall $\theta^0_t=\phi^0_{tt}$),
\begin{align}
\int_0^t \ep\|A^{-1/2}_N\theta^0_t(s)\|^2 ds \le \ep Q(R,m).  \label{usc-20.2}
\end{align}
Then integrating \eqref{usc-19.1} over $(0,t)$, we obtain, for all $t\in[0,T],$
\begin{align}
\ep\|\tilde\theta(t)\|^2 + c_0\|\tilde\phi(t)\|^2 & - (J*\tilde\phi(t),\tilde\phi(t)) + \int_0^t \left( 2\|\tilde\phi_t(s)\|^2_{V'} + \alpha\|\tilde\phi_t(s)\|^2 + \ep\|\tilde\theta(s)\|^2_V \right) ds  \notag \\ 
& \le \ep\|\tilde\theta(0)\|^2 + C|\Omega||\langle \tilde\phi \rangle| \tilde\mu_* + \ep\langle\tilde\theta\rangle^2 + (\alpha+\ep) Q(R,m,T).  \label{diff-19}
\end{align}
Here we apply \eqref{scl-6}-\eqref{usc-e} and \eqref{extension} to reduce \eqref{diff-19} into
\begin{align}
\ep\|\tilde\theta(t)\|^2 + c_0\|\tilde\phi(t)\|^2 & - (J*\tilde\phi(t),\tilde\phi(t)) + \int_0^t \left( 2\|\tilde\phi_t(s)\|^2_{V'} + \alpha\|\tilde\phi_t(s)\|^2 + \ep\|\tilde\theta(s)\|^2_V \right) ds  \notag \\ 
& \le (\alpha+\ep) Q(R,m,T).  \notag
\end{align}

It is here where we employ the assumption (H6), then there holds
\begin{align}
\ep\|\tilde\theta(t)\|^2 + (c_0-c_J)\|\tilde\phi(t)\|^2 & + \int_0^t \left( 2\|\tilde\phi_t(s)\|^2_{V'} + \alpha\|\tilde\phi_t(s)\|^2 + \ep\|\tilde\theta(s)\|^2_V \right) ds  \notag \\ 
& \le (\alpha+\ep) Q(R,m,T).  \notag
\end{align}
By applying the embedding $V'\hookleftarrow H$ and using the fact that $\alpha\in(0,1]$, we find 
\begin{align}
(c_0-c_J)\|\tilde\phi\|^2 & \ge \frac{c_0-c_J}{2}\|\tilde\phi\|^2 + \frac{c_0-c_J}{2}\alpha\|\tilde\phi\|^2  \notag \\ 
& \ge \frac{c_0-c_J}{2}C^{-2}_\Omega\|\tilde\phi\|^2_{V'} + \frac{c_0-c_J}{2}\alpha\|\tilde\phi\|^2  \notag \\
& \ge c\left( \|\tilde\phi\|^2_{V'} + \alpha\|\tilde\phi\|^2 \right),  \notag
\end{align}
for some suitably small constant $c>0$ independent of $\alpha$ and $\ep$.
With this we arrive at the estimate \eqref{rob-1} as claimed. 
This completes the proof.
\end{proof}

The following is the main result of this section.

\begin{theorem}  \label{t:robustness}
Under the hypotheses of Lemma \ref{t:rob-1}, the family of sets $(\mathbb{A}^{\alpha,\ep})_{\alpha,\ep\in[0,1]}$ satisfies the following upper-semicontinuity estimate in the topology of $\mathbb{H}^{\alpha,\ep}_m$, 
\begin{align}
{\rm{dist}}_{\mathbb{H}^{\alpha,\ep}_m}(\mathbb{A}^{\alpha,\ep},\mathbb{A}^{0,0}) \le (\alpha+\ep)^{1/2} Q(R,m,T),  \label{main-0}
\end{align}
for some positive increasing function $Q$ and where $R>0$ is the uniform bound on $\mathcal{A}^{\alpha,\ep}$ in $\mathbb{H}^{\alpha,\ep}_m$ (this bound is given by the radius of $\mathcal{B}^{\alpha,\ep}_0$, also recall Remark \ref{r:uniform-radius}).
\end{theorem}

\begin{proof}
To begin,
\[
{\rm{dist}}_{\mathbb{H}^{\alpha,\ep}_m}(\mathbb{A}^{\alpha,\ep},\mathbb{A}^{0,0}) =  \sup_{a\in\mathcal{A}^{\alpha,\ep}}\inf_{b\in\mathcal{LA}^{0,0}}\|a-b\|_{\mathbb{H}^{\alpha,\ep}_m}.
\]
Fix $t\in[0,T]$ and $\xi\in \mathbb{A}^{\alpha,\ep}$ so that $a=S_{\alpha,\ep}(t)\xi\in\mathbb{A}^{\alpha,\ep}$. 
Then
\begin{align}
\inf_{b\in\mathcal{LA}^{0,0}}\|a-b\|_{\mathbb{H}^{0,0}_m} & = \inf_{\substack{\tau\in [0,T] \\ \zeta\in \mathcal{A}^{0,0}}}\|S_{\alpha,\ep}(t)\xi-\mathcal{L}S_{0,0}(\tau)\zeta\|_{\mathbb{H}^{\alpha,\ep}_m}  \notag \\
& \le \inf_{\zeta\in \mathcal{A}^{0,0}}\|S_{\alpha,\ep}(t)\xi-\mathcal{L}S_{0,0}(t)\zeta\|_{\mathbb{H}^{\alpha,\ep}_m}.  \notag
\end{align}
Since $S_{\alpha,\ep}(t)\xi=a$, 
\begin{align}
\sup_{\xi\in \mathcal{A}^{\alpha,\ep}}\inf_{b\in\mathcal{L}\mathcal{A}^{0,0}}\|S_{\alpha,\ep}(t)\xi-b\|_{\mathbb{H}^{\alpha,\ep}_m} & \leq \sup_{\xi\in \mathcal{A}^{\alpha,\ep}}\inf_{\zeta\in \mathcal{A}^{0,0}}\|S_{\alpha,\ep}(t)\xi-\mathcal{L}S_{0,0}(t)\zeta\|_{\mathbb{H}^{\alpha,\ep}_m}   \notag \\ 
& ={\rm{dist}}_{\mathbb{H}^{\alpha,\ep}_m}(S_{\alpha,\ep}(t)\mathcal{A}^{\alpha,\ep},\mathcal{L}S_{0,0}(t)\mathcal{A}^{0,0})  \notag \\
& \leq \max_{t\in [0,T]}{\rm{dist}}_{\mathbb{H}^{\alpha,\ep}_m}(S_{\alpha,\ep}(t)\mathcal{A}^{\alpha,\ep},\mathcal{L}S_{0,0}(t)\mathcal{A}^{0,0}). \notag
\end{align}
Thus, 
\begin{align}
\sup_{t\in [0,T]}\sup_{\xi\in \mathcal{A}^{\alpha,\ep}}\inf_{b\in\mathcal{L}\mathcal{A}^{0,0}} \|S_{\alpha,\ep}(t)\xi-b\|_{\mathbb{H}^{\alpha,\ep}_m} \leq \max_{t\in [0,T]}{\rm{dist}}_{\mathbb{H}^{\alpha,\ep}_m}(S_{\alpha,\ep}(t)\mathcal{A}^{\alpha,\ep},\mathcal{L}S_{0,0}(t)\mathcal{A}^{0,0}),  \notag
\end{align}
and 
\begin{align}
\sup_{a\in\mathbb{H}^{\alpha,\ep}_m}\inf_{b\in\mathcal{L}\mathcal{A}^{0,0}}\|a-b\|_{\mathbb{H}^{\alpha,\ep}_m} & \leq \sup_{t\in [0,T]}\sup_{\xi\in \mathcal{A}^{\alpha,\ep}}\inf_{b\in\mathcal{L}\mathcal{A}^{0,0}} \|S_{\alpha,\ep}(t)\xi-b\|_{\mathbb{H}^{\alpha,\ep}_m}  \notag \\
& \leq \max_{t\in [0,T]}{\rm{dist}}_{\mathbb{H}^{\alpha,\ep}_m}(S_{\alpha,\ep}(t) \mathcal{A}^{\alpha,\ep},\mathcal{L}S_{0,0}(t)\mathcal{A}^{0,0})  \notag \\
& \leq \max_{t\in [0,T]} \sup_{\xi\in \mathcal{A}^{\alpha,\ep}}\inf_{\zeta\in \mathcal{A}^{0,0}}\|S_{\alpha,\ep}(t)\xi-\mathcal{L}S_{0,0}(t)\zeta\|_{\mathbb{H}^{\alpha,\ep}_m}.  \notag
\end{align}

The norm is then expanded
\begin{align}
\|S_{\alpha,\ep}(t)\xi-\mathcal{L}S_{0,0}(t)\zeta\|_{\mathbb{H}^{\alpha,\ep}_m} \leq \|S_{\alpha,\ep}(t)\xi - \mathcal{L}S_{0,0}(t)\Pi\xi\|_{\mathbb{H}^{\alpha,\ep}_m} + \|\mathcal{L}S_{0,0}(t)\Pi\xi-\mathcal{L}S_{0,0}(t)\zeta\|_{\mathbb{H}^{\alpha,\ep}_m} \label{4-1}
\end{align}
so that by \eqref{rob-1} we know
\[
\|S_{\alpha,\ep}(t)\xi-\mathcal{L}S_{0,0}(t)\Pi\xi\|_{\mathbb{H}^{\alpha,\ep}_m} \le (\alpha+\ep)^{1/2} Q(R,m,T).
\]
Expand the square of the norm on the right-hand side of (\ref{4-1}) to obtain, for $\Pi\xi=\Pi(\xi_1,\xi_2)=\xi_1\in \mathbb{H}^{0,0}_m\subset H$ and $\zeta\in \mathbb{H}^{0,0}_m\subset H$,
\begin{align}
& \|\mathcal{L}S_{0,0}(t)\Pi\xi-\mathcal{L}S_{0,0}(t)\zeta\|^2_{\mathbb{H}^{\alpha,\ep}_m}  \notag \\ 
& = \|S_{0,0}(t)\Pi\xi-S_{0,0}(t)\zeta\|^2_{V'} + \alpha\|S_{0,0}(t)\Pi\xi-S_{0,0}(t)\zeta\|^2 + \ep\|\mathcal{M}(S_{0,0}(t)\Pi\xi)-\mathcal{M}(S_{0,0}(t)\zeta)\|^2.  \label{triangle-r}
\end{align}
Recall from \eqref{extension}, the terms $\mathcal{M}S_{0,0}(t)\Pi\xi$ and $\mathcal{M}S_{0,0}(t)\zeta$ can be expressed in terms of the chemical potential $\mu^0$. 
By the continuous embedding $H\hookrightarrow V'$, and by the local Lipschitz continuity of the maps $\mathcal{M}$ (this is possible thanks to (H2)) and $S_{0,0}$ on $\mathbb{H}^{0,0}_m$ (cf. \cite[Proposition 2.13]{Gal&Grasselli14}), then (\ref{triangle-r}) can be estimated by 
\[
\|\mathcal{L}S_{0,0}(t)\Pi\xi-\mathcal{L}S_{0,0}(t)\zeta\|^2_{\mathbb{H}^{\alpha,\ep}_m} \le Q(R,m,T)(1+\alpha+\ep)\|\Pi\xi-\zeta\|^2_{\mathbb{H}^{0,0}_m}.
\]
Hence, (\ref{4-1}) becomes
\begin{align}
\|S_{\alpha,\ep}(t)\xi-\mathcal{L}S_{0,0}(t)\zeta\|_{\mathbb{H}^{\alpha,\ep}_m} \le Q(R,m,T) \left( (\alpha+\ep)^{1/2} + (1+\alpha+\ep)^{1/2}\|\Pi\xi-\zeta\|_{\mathbb{H}^{0,0}_m} \right),  \notag
\end{align}
and 
\begin{align}
\inf_{\zeta\in \mathcal{A}^{0,0}}\|S_{\alpha,\ep}(t)\xi-\mathcal{L}S_{0,0}(t)\zeta\|_{\mathbb{H}^{\alpha,\ep}_m} \le Q(R,m,T) \left( (\alpha+\ep)^{1/2} + (1+\alpha+\ep)^{1/2}\|\Pi\xi-\zeta\|_{\mathbb{H}^{0,0}_m} \right).  \notag
\end{align}
Since $\Pi\xi\in\Pi \mathcal{A}^{\alpha,\ep}=\mathcal{A}^{0,0}$, then it is possible to choose $\zeta\in \mathcal{A}^{0,0}$ to be $\zeta=\Pi\xi$. 
Therefore, 
\begin{align}
{\rm{dist}}_{\mathbb{H}^{\alpha,\ep}_m}(\mathbb{A}^{\alpha,\ep},\mathbb{A}^{0,0}) & = \sup_{\alpha\in \mathcal{A}^{\alpha,\ep}}\inf_{\zeta\in \mathcal{A}^{0,0}}\|S_{\alpha,\ep}(t)\xi-\mathcal{L}S_{0,0}(t)\zeta\|_{\mathbb{H}^{\alpha,\ep}_m}  \notag \\ 
& \le Q(R,m,T) (\alpha+\ep)^{1/2}.
\end{align}
This establishes the estimate \eqref{main-0} and finishes the proof. 
\end{proof}

\section{Conclusions and further remarks}

In this article we have shown that the family of global attractors generated by a relaxation of Problem P$_{0,0}$, given by Problem P$_{\alpha,\ep}$, is upper-semicontinuous as the perturbation parameters vanish.
With this we verify a rather classical result going back to \cite{Hale&Raugel88}.
We also establish explicit control over the semidistances in explicit terms of the parameters despite the many difficulties due to the presence of the nonlocal diffusion term on the order parameter $\phi$.
Two essential results that lead to this type of continuity result are the continuous dependence estimate in Proposition \ref{t:cont-dep} (because it is instrumental in obtaining Lemma \ref{t:rob-1}) and Lemma \ref{t:du}, whereby the difficulty of controlling $\phi_{tt}$ in the nonviscous isothermal Problem P$_{0,0}$ becomes apparent.
It seems that this type of upper-semicontinuity result for nonlocal Cahn-Hilliard equations is the first of its kind.

Also concerning the global attractors, some interesting future work would include determining whether the (fractal) dimension of the global attractors found here is finite and {\em{independent}} of $\alpha$ and $\ep$.
Hence, we should also examine the existence of an exponential attractor for Problem P$_{\alpha,\ep}$, and naturally, its basin of attraction.
With that result, we could seek a {\em{robustness}} result for the family of exponential attractors (that is, the upper- and lower-semicontinuity of the attractors with respect to $\alpha$ and $\ep$).
Examining problems related to stability (and hence the approximation of the longterm behavior of a relaxation problem to the associated limit problem) may prove to be an important source of further work on nonlocal Cahn-Hilliard and nonlocal phase field models.
Concerning another perturbation/relaxation problem, it might be interesting to see if comparable results appear in the model problems described by \eqref{mot-1}.

Of course, some future work may examine several variants to the current model. 
Such variants may include a convection term that accounts for the effects of an averaged (fluid) velocity field, which naturally couples with a nonisothermal Navier-Stokes equation (on the former, see for example \cite{Porta-Grasselli-2014}).
Indeed, one may include nonconstant mobility in the nonlocal Cahn-Hilliard equation (cf. e.g. \cite{Frigeri-Grasselli-Rocca_2014}). 
It may be interesting to generalize the coupled heat equation to a Coleman-Gurtin type equation. 
Also, one may examine the associated nonlocal phase-field model \eqref{mot-1}, and the effects of generalizing the heat equation along the lines of \cite{Herrera&Pavon02,Joseph-Preziosi-89,Joseph-Preziosi-90,GPS07} where Fourier's law is replaced with a Maxwell-Cattaneo law because in this more realistic setting, ``disturbances'' travel at a {\em{finite}} speed. 

It would also be interesting to study the nonlocal variant of the Cahn-Hilliard and phase-field equations by introducing relevant dynamic boundary conditions (again, see \cite{Gal&Miranville09}).
In this case, several interesting difficulties may arise concerning the regularity of solutions because, typically in applications, $H^1(\Omega)$ regularity (or better) is sought in order to define the trace of the solution; recall, $trace:H^s(\Omega)\rightarrow H^{s-1/2}(\Gamma).$ 
Additionally, we should study the case when the potential is singular (see hypotheses in \cite[Section 3]{Gal&Grasselli14}, for example).

\appendix
\section{}  \label{s:append}

\begin{proof}[Proof of Lemma \ref{t:diss-1}]
We give a formal calculation that can be justified by a suitable Faedo-Galerkin approximation based on the proof of Theorem \ref{t:existence} above.
Let $M_0:=\langle\phi_0\rangle$ and $N_0:=\langle\theta_0\rangle$.
Multiply \eqref{rel-1}-\eqref{rel-3} by, $A_N^{-1}\phi_t$, $\phi_t$, and $\hat\theta:=\theta-N_0$, respectively, then integrate over $\Omega$, applying \eqref{NLr-1} (since $\phi_t=\phi_t-\langle\phi_t\rangle$ belongs to $V_0'$; recall \eqref{con-mass}), and sum the resulting identities to arrive at the differential identity, which holds for almost all $t\ge0$,
\begin{align}
& \frac{d}{dt} \left\{ \|\sqrt{a}\phi\|^2 + \ep\|\hat\theta\|^2 + 2(F(\phi),1) - (J*\phi,\phi) \right\} + 2\|\phi_t\|^2_{V'} + 2\alpha\|\phi_t\|^2 + 2\|\nabla\theta\|^2 = 0.  \label{ap-2}
\end{align}
Let $\hat\phi:=\phi-M_0$.
We further multiply \eqref{rel-1}-\eqref{rel-2} by, $2\xi A_N^{-1}\hat\phi$ and $2\xi\hat\phi$, respectively, in $H$, where $\xi>0$ is to be determined below.
Observe $\langle\hat\phi\rangle=0$ and $\|\hat\phi\|^2=\|\phi\|^2-M_0^2|\Omega|.$
This yields, for almost all $t\ge0,$
\begin{align}
& \frac{d}{dt} \left\{ \xi\|\hat\phi\|^2_{V'} + \xi\alpha\|\hat\phi\|^2 \right\} + 2\xi\|\sqrt{a}\hat\phi\|^2 + 2\xi(F'(\phi),\hat\phi)  \notag \\ 
& = 2\xi(J*\phi,\hat\phi) + 2\xi\delta(\theta,\hat\phi) - 2\xi M_0(a,\hat\phi).  \label{ap-3}
\end{align}
Together, \eqref{ap-2} and \eqref{ap-3} make the differential identity,
\begin{align}
& \frac{d}{dt} \left\{ \xi\|\hat\phi\|^2_{V'} + \xi\alpha\|\hat\phi\|^2 + \|\sqrt{a}\phi\|^2 + \ep\|\hat\theta\|^2 + 2(F(\phi),1) - (J*\phi,\hat\phi) \right\}  \notag \\ 
& + 2\|\phi_t\|^2_{V'} + 2\alpha\|\phi_t\|^2 + 2\xi\|\sqrt{a}\hat\phi\|^2 + 2\|\nabla\theta\|^2 + 2\xi(F'(\phi),\hat\phi)  \notag \\
& = 2\xi(J*\phi,\hat\phi) + 2\xi\delta(\theta,\hat\phi) - 2\xi M_0(a,\hat\phi).  \label{ap-4}
\end{align}
Introduce the functional defined by, for all $t\ge0$ and $\xi>0$,
\begin{equation}  \label{ap-5}
E(t):=\xi\|\hat\phi(t)\|^2_{V'} + \xi\alpha\|\hat\phi(t)\|^2 + \|\sqrt{a}\phi(t)\|^2 + \ep\|\hat\theta(t)\|^2 + 2(F(\phi(t)),1) - (J*\phi,\hat\phi) + C_F.
\end{equation}
(Observe, $E(t) = 2\mathcal{E}_\ep(t) + \xi\|\hat\phi(t)\|^2_{V'} + \xi\alpha\|\hat\phi(t)\|^2 + C_F$.)
Because of assumption (H3) and the assumption that $F(\phi_0)\in L^1(\Omega)$, we know 
\begin{equation}  \label{Fcons-5}
2(F(\phi),1) - (J*\phi,\hat\phi) \ge (2c_1-2c_1)\|\hat\phi\|^2+2c_1 M_0^2|\Omega|-2c_2|\Omega|,
\end{equation}
thus the constant $C_F$ may be chosen sufficiently large to insure $E(t)$ is non-negative for all $t\ge0$, $\alpha\in(0,1]$, $\ep\in(0,1]$, and $\xi>0.$
Then we rewrite \eqref{ap-4} as, 
\begin{align}
& \frac{d}{dt} E + \tau E = H,  \label{ap-6}
\end{align}
for some $0<\tau<\xi,$ and where 
\begin{align}
H(t) & := \tau\xi\|\hat\phi(t)\|^2_{V'} + \tau\xi\alpha\|\hat\phi(t)\|^2 + \tau\|\sqrt{a}\phi(t)\|^2 + \tau\ep\|\hat\theta(t)\|^2 + 2\tau(F(\phi(t)),1) - \tau(J*\phi,\hat\phi) + \tau C_F  \notag \\ 
& - 2\|\phi_t(t)\|^2_{V'} - 2\alpha\|\phi_t(t)\|^2 - 2\xi\|\sqrt{a}\hat\phi(t)\|^2 - 2\|\nabla\theta(t)\|^2 - 2\xi(F'(\phi(t)),\hat\phi(t))  \notag \\
& + 2\xi(J*\phi(t),\hat\phi(t)) + 2\xi\delta(\theta(t),\hat\phi(t)) - 2\xi M_0(a,\hat\phi(t)).  \label{ap-7}
\end{align}
Estimating the products on the right-hand side using the assumptions (H1)-(H3) as well as Young's inequality for convolutions (cf. e.g. \cite[Corollary 2.25]{Adams&Fournier03}), and the Poincar\'{e}-type inequality \eqref{Poincare2} yields (and recall $\delta\in(0,\delta_0]$), 
\begin{align}
2\xi(J*\phi,\hat\phi) & \le 2\xi \|J*\phi\| \|\hat\phi \|  \notag \\ 
& \le 2\xi c_J \|\hat\phi\|^2 + M_0^2\|a\|_\infty^2 + \xi^2\|\hat\phi\|^2,  \label{ap-8}
\end{align}
\begin{align}
2\xi\delta(\theta,\hat\phi) & \le 2\xi\delta_0 \|\theta\|\|\hat\phi\|  \notag \\ 
& \le \xi\delta_0^2\|\theta\|^2 + \xi\|\hat\phi\|^2  \notag \\
& \le 2\xi\delta_0^2\lambda_\Omega\|\nabla\theta\|^2 + 2\xi\delta_0^2|\Omega|N_0^2 + \xi\|\hat\phi\|^2,  \label{ap-10}
\end{align}
and
\begin{align}
-2\xi M_0(a,\hat\phi) & \le 2\xi M_0 \|a\| \|\hat\phi\|  \notag \\ 
& = 2\xi M_0\|J*1\|\|\hat\phi\|  \notag \\
& \le 2\xi M_0 c_J |\Omega|^{1/2}\|\hat\phi\|  \notag \\
& \le M_0^2 c_J^2 |\Omega| + \xi^2 \|\hat\phi\|^2.  \label{ap-11}
\end{align}
With assumption (H3) we now consider, with the aid of \eqref{Fcons-1}-\eqref{Fcons-3} (setting $m=M_0$),
\begin{align}
2\tau(F(\phi),1) - 2\xi(F'(\phi),\hat\phi) & = -2\tau\left( (F'(\phi),\hat\phi)-(F(\phi),1) \right) - 2(\xi-\tau)(F'(\phi),\hat\phi)  \notag \\
& = -2\tau(F'(\phi)\hat\phi-F(\phi),1)-2(\xi-\tau)(F'(\phi),\hat\phi)  \notag \\ 
& \le 2\tau c_9|\Omega| + 2\tau c_{10}\|\hat\phi\|^2 - (\xi-\tau)(|F(\phi)|,1) + 2(\xi-\tau)c_{11} + (\xi-\tau)c_{12}.  \label{ap-12} 
\end{align}
By (H1) again, we find that for a fixed $0<a_0<\essinf_\Omega a(x)$ (this is where we need the slightly stricter version of (H1)), there holds 
\[
a_0\|\hat\phi\|^2 \le \|\sqrt{a}\hat\phi\|^2.
\]
Moreover, due to the continuous embedding $H\hookrightarrow V',$ there is a constant, which we denote $C_\Omega>0$, so that $C_\Omega^{-2}\|\hat\phi\|^2_{V'}\le\|\hat\phi\|^2$ (cf. e.g. \cite[p. 243, Equation (6.7)]{Milani&Koksch05}), and, now with $0<\xi<1,$ 
\begin{align}
-2\xi\|\sqrt{a}\hat\phi\|^2 & \le -\frac{a_0}{2} C^{-2}_\Omega\|\hat\phi\|^2_{V'} - \frac{a_0}{2}\|\hat\phi\|^2 - \xi\|\sqrt{a}\hat\phi\|^2.  \label{ap-13}
\end{align}
Also observe that, using the Poincar\'{e}-type inequality \eqref{Poincare} again, we have
\begin{align}
-\left( 2 - 2\xi\delta_0^2\lambda_\Omega \right)\|\nabla\theta\|^2 \le -\left( 1 - 2\xi\delta_0^2\lambda_\Omega \right)\|\nabla\theta\|^2 - \frac{1}{\lambda_\Omega}\|\hat\theta\|^2.  \label{ap-13.5}
\end{align} 
Combining \eqref{ap-7}-\eqref{ap-13.5} yields, 
\begin{align}
H & \le \left( \tau\xi - \frac{a_0}{2} C^{-2}_\Omega \right)\|\hat\phi\|^2_{V'} + \left( \tau\xi\alpha + 2\xi c_J + \xi + 2\xi^2 + 2\tau c_{10} - \frac{a_0}{2} \right)\|\hat\phi\|^2  \notag \\ 
& + \left( \tau - \frac{\xi}{2} \right) \|\sqrt{a}\phi\|^2 + \left( \tau\ep - \frac{1}{\lambda_\Omega} \right)\|\hat\theta\|^2 - (\xi-\tau)(|F(\phi)|,1)  \notag \\ 
& - 2\|\phi_t\|^2_{V'} - 2\alpha\|\phi_t\|^2 - \left(1-2\xi\delta_0^2\lambda_\Omega \right) \|\nabla\theta\|^2  \notag \\
& + \tau C_F + M_0^2 c_J^2|\Omega| + 2\xi\delta_0^2|\Omega|N_0^2 + 2\tau c_9|\Omega|  \notag \\
& + 2(\xi-\tau)c_{11} + (\xi-\tau)c_{12} + \xi M_0^2|\Omega|(\langle a \rangle - a_0) + M_0^2\|a\|_\infty^2.  \label{ap-14}
\end{align}
We should note that the additional constants in $a$ on the right-hand side of \eqref{ap-14} is due to the fact that
\[
-\xi\|\sqrt{a}\hat\phi\|^2 \ge -\xi\|\sqrt{a}\phi\|^2 - \xi M_0^2|\Omega|(\langle a \rangle - a_0).
\]
Inserting \eqref{ap-14} into \eqref{ap-6} produces the differential inequality (this is where we use the condition that $0<\alpha\le1$ and $0<\ep\le1$), 
\begin{align}
& \frac{d}{dt} E + 2\|\phi_t\|^2_{V'} + 2\alpha\|\phi_t\|^2 + \left( 1-2\xi\delta_0^2\lambda_\Omega \right)\|\theta\|^2_{V}  \notag \\ 
& + \frac{a_0}{4}C^{-2}_\Omega\|\hat\phi\|^2_{V'} + \left( \frac{a_0}{2} - 2\xi c_J - \xi - 2\xi^2 - 2\tau c_{10} \right)\alpha\|\hat\phi\|^2  \notag \\ 
& + \xi\|\sqrt{a}\phi\|^2 + \frac{1}{\lambda_\Omega}\ep\|\hat\theta\|^2 + (\xi - \tau)(F(\phi),1) + \tau C_F  \notag \\ 
& \le \tau C_F + M_0^2 c_J^2|\Omega| + 2\xi\delta_0^2|\Omega|N_0^2 + \left( 1-2\xi\delta_0^2\lambda_\Omega \right)N^2_0  \notag \\ 
& + 2\tau c_9|\Omega| + 2(\xi-\tau)c_{11} + (\xi-\tau)c_{12} + \xi M_0^2|\Omega|(\langle a \rangle - a_0) + M_0^2\|a\|_\infty^2.  \notag 
\end{align}
The extra term with $N_0$ now appearing on the right-hand side is used to make the $V$ norm in $\theta$.
Now we easily see that there are $0<\tau<\xi<1$ so that 
\[
\nu_3=\nu_3(\delta_0,J,\Omega):=\min\left\{ 1-2\xi\delta_0^2\lambda_\Omega, \frac{a_0}{2} - 2\xi c_J - \xi - 2\xi^2 - 2\tau c_{10} \right\}>0.
\]
Now there holds, for almost all $t\ge0,$
\begin{align}
& \frac{d}{dt} E + \nu_3 E + \|\phi_t\|^2_{V'} + 2\alpha\|\phi_t\|^2 + \nu_3\|\theta\|^2_{V} \le Q(m).  \label{ap-15}
\end{align}
Neglecting the normed terms $\|\phi_t\|^2_{V'} + 2\alpha\|\phi_t\|^2 + \nu_3\|\theta\|^2_{V}$, then employing a Gr\"{o}nwall inequality yields, for all $t\ge0,$
\begin{equation}
E(t) \le e^{-\nu_3 t}E(0) + \frac{1}{\nu_3}Q(m).  \label{ap-16}
\end{equation}
Recall that $F(\phi_0)\in L^1(\Omega)$ by assumption, so now we easily arrive at 
\begin{align}
& \|\hat\phi(t)\|^2_{V'} + \alpha\|\hat\phi(t)\|^2 + \|\sqrt{a}\phi(t)\|^2 + \|\hat\theta(t)\|^2 + (F(\phi(t)),1) - (J*\phi(t),\hat\phi(t))  \notag \\ 
& \le E(0)e^{-\nu_3t} + \frac{1}{\nu_3}Q(m).  \label{ap-17}
\end{align}
Also, by neglecting the positive term $\nu_3E$ in \eqref{ap-15} and integrating this time over $(t,t+1)$, we find, with \eqref{ap-16}, for all $t\ge0$,
\begin{equation}
\int_t^{t+1}\left( \|\phi_t(s)\|^2_{V'} + \alpha\|\phi_t(s)\|^2 + \|\theta(s)\|^2_{V} \right)ds \le E(0)e^{-\nu_3t} + \left( \frac{1}{\nu_3} + 1 \right)Q(m).  \label{ap-18}
\end{equation}
Together, \eqref{ap-17} and \eqref{ap-18} establish \eqref{ap-1}.

The existence of the set $\mathcal{B}^{\alpha,\ep}_0$ described in \eqref{abs-set} follows directly from the dissipation estimate \eqref{ap-1}; indeed, (cf. e.g. \cite{Babin&Vishik92}).
To see why $\mathcal{B}^{\alpha,\ep}_0$ is absorbing, consider any nonempty bounded subset $B$ in $\mathbb{H}^{\alpha,\ep}_m\setminus\mathcal{B}^{\alpha,\ep}_0$.
Then we have that $\mathcal{S}_{\alpha,\ep}(t)B\subseteq\mathcal{B}^{\alpha,\ep}_0$, in $\mathbb{H}^{\alpha,\ep}_m$, for all $t\ge t_0$, where
\begin{equation}  \label{time-t0}  
t_0 := \max\left\{ \frac{1}{\nu_3}\ln(E(0)),0 \right\}.
\end{equation}
This completes the proof.
\end{proof}

\bigskip

\bibliographystyle{amsplain}
\providecommand{\bysame}{\leavevmode\hbox to3em{\hrulefill}\thinspace}
\providecommand{\MR}{\relax\ifhmode\unskip\space\fi MR }
\providecommand{\MRhref}[2]{%
  \href{http://www.ams.org/mathscinet-getitem?mr=#1}{#2}
}
\providecommand{\href}[2]{#2}

\end{document}